\documentclass[a4paper]{amsart}
\usepackage{verbatim}
\usepackage[utf8]{inputenc}
\usepackage{eqparbox}
\newcommand{\eqmathbox}[2][eq]{\eqmakebox[#1]{$\displaystyle #2$}}

\usepackage[toc]{appendix}
\usepackage{bm}
\usepackage[T1]{fontenc}
\usepackage{graphicx}
\usepackage{enumerate}
\usepackage{amsmath,amsfonts,amssymb}
\usepackage{color}
\def\loc{\operatorname{loc}}
\usepackage{cite}
\usepackage{ latexsym }
\definecolor{citation}{rgb}{0.11,0.67,0.84}
\definecolor{formula}{rgb}{0.1,0.2,0.6}
\definecolor{url}{rgb}{0.11,0.67,0.84}
\usepackage{pgf,tikz}
\usepackage{mathrsfs}
\usepackage{fancyhdr}
\usepackage{dutchcal}

\usepackage{tikz}
\usepackage{amsmath}

\newcommand{\medint}{-\kern -,375cm\int}

\newcommand{\medintinrigo}{-\kern -,315cm\int}
\makeatletter
\newcommand{\linethrough}{\mathpalette\@thickbar}
\newcommand{\@thickbar}[2]{{#1\mkern0mu\vbox{
    \sbox\z@{$#1#2\mkern-0.5mu$}%
    \dimen@=\dimexpr\ht\tw@-\ht\z@+2\p@\relax % The +2 represents the vertical shift of the line.
    \hrule\@height0.5\p@ % The 0.5 represent the thickness on the line.
    \vskip\dimen@
    \box\z@}}
}
\makeatother

\newtheorem{theorem}{Theorem}[section]
\newtheorem{lemma}[theorem]{Lemma}
\newtheorem{proposition}[theorem]{Proposition}
\newtheorem{remark}{Remark}[section] 
\newtheorem{corollary}[theorem]{Corollary}
\newtheorem{definition}[theorem]{Definition}
\theoremstyle{plain}  % Testo normale (non corsivo)
\usepackage{xcolor} % per usare colori personalizzati
\usepackage[
  colorlinks=true,   
  linkcolor=blue,   
  citecolor=blue,    
  urlcolor=blue      
]{hyperref}
\numberwithin{equation}{section}
\usepackage{amsthm}
\usepackage{amsthm}

\usepackage{dsfont}

\newcommand{\reqnomode}{\tagsleft@false}

\usepackage{hyperref}

\def\dx{\,{\rm d}x}

\def\dt{\,{\rm d}t}
\def\dy{\,{\rm d}y}

\def \d{\,{\rm d}}

\def\supp{\,{\rm supp}}

\allowdisplaybreaks
\makeatletter
\DeclareRobustCommand*{\bfseries}{%
  \not@math@alphabet\bfseries\mathbf
  \fontseries\bfdefault\selectfont
  \boldmath
}

\makeatother

\newlength{\defbaselineskip}
\newcommand{\mint}{\mathop{\int\hskip -1,05em -\, \!\!\!}\nolimits}

\newcommand{\R}{\mathbb{R}}

\newcommand{\tu}{\tilde{u}}
\newcommand{\tv}{\tilde{v}}

\newcommand{\bb}{\tx{B}}
\newcommand{\br}{\tx{B}_{r}}
\newcommand{\brr}{\tx{B}_{2r}}

\usepackage{enumitem}
\newcommand{\dd}{\mathrm{d}}

\newcommand{\mm}{\mathsf{M}}

\newcommand{\ff}{\tx{H}}

\newcommand{\g}{\tx{g}}

\newcommand{\rr}{\varrho}
\newcommand{\snr}[1]{\lvert #1\rvert}
\newcommand{\nr}[1]{\lVert #1 \rVert}

\newcommand{\tx}[1]{\textnormal{\texttt{#1}}}

\def\loc{\operatorname{loc}}
\def\eqn#1$$#2$${\begin{equation}\label#1#2\end{equation}}

\def\supp{\,{\rm supp }}

\DeclareMathOperator*{\essinf}{ess\,inf}

\def\Xint#1{\mathchoice
   {\XXint\displaystyle\textstyle{#1}}%
   {\XXint\textstyle\scriptstyle{#1}}%
   {\XXint\scriptstyle\scriptscriptstyle{#1}}%
   {\XXint\scriptscriptstyle\scriptscriptstyle{#1}}%
   \!\int}
\def\XXint#1#2#3{{\setbox0=\hbox{$#1{#2#3}{\int}$}
     \vcenter{\hbox{$#2#3$}}\kern-.5\wd0}}

\def\dashint{\Xint-}

\def\Xint#1{\mathchoice
	{\XXint\displaystyle\textstyle{#1}}%
	{\XXint\textstyle\scriptstyle{#1}}%
	{\XXint\scriptstyle\scriptscriptstyle{#1}}%
	{\XXint\scriptscriptstyle\scriptscriptstyle{#1}}%
	\!\int}
\def\XXint#1#2#3{{\setbox0=\hbox{$#1{#2#3}{\int}$ }
		\vcenter{\hbox{$#2#3$ }}\kern-.6\wd0}}

\def\dashint{\Xint-}
\title{Partial regularity for Double Phase obstacle problems}
\author[De Filippis]{Filomena De Filippis}  \address{Filomena De Filippis\\Dipartimento SMFI, Universit\`a di Parma, Viale delle Scienze 53/a, Campus, 43124 Parma, Italy \\ ORCID ID: 0000-0002-2784-1411}
\email{\url{filomena.defilippis@unipr.it}}
\author[Nastasi]{Antonella Nastasi}  \address{Antonella Nastasi\\Department of Engineering, University of Palermo, Viale delle Scienze, 90128, Palermo, Italy\\
 ORCID ID: 0000-0003-1589-2235}
\email{\url{antonella.nastasi@unipa.it}}
\author[Pacchiano Camacho]{Cintia Pacchiano Camacho}  \address{Cintia Pacchiano Camacho\\Instituto de Matem\'aticas, Unidad Cuernavaca, Universidad Nacional Aut\'onoma de M\'exico, Av. Universidad, 62210, Cuernavaca, Morelos, Mexico\\
ORCID ID: 0009-0004-6210-4013}
\email{\url{cintia.pacchiano@im.unam.mx}}

\begin{document}

\subjclass[2020]{35J60, 49N60, 35J50\vspace{1mm}}

\begin{abstract}
We investigate  partial regularity for vector valued local minimizers of double phase functionals, under vectorial obstacle type constraints satisfying appropriate topological properties.
\end{abstract}

\maketitle
\begin{center}
\begin{minipage}{12cm}
  \small
  \tableofcontents
\end{minipage}
\end{center}

\section{Introduction}
\noindent
In this paper, we study  gradient partial regularity for a constrained vector valued function \(u\colon \Omega \to \mathbb{R}^N\), $N \geq 3$, which locally minimizes a double phase energy of the type
\eqn{1f1}
$$ \mathcal{G}(u,\Omega)  := \int_{\Omega} \tx{g}(x,u,Du) \mathrm{d}x, \qquad |z|^p \lesssim \tx{g}(x,u,z) \lesssim 1+ |z|^q, \qquad \text{with } 1<p<q,$$
and whose values lie on a target $\mm$ that  satisfies one of the following conditions:
\begin{enumerate}[label=\tx{\alph*.}, ref=\tx{\alph*}]
    \item\label{item:a} $\mm$ is a bounded open subset of $\R^N$ with boundary $ \partial \mm \in C^3$ ;
    \item\label{item:b}$\mm$ is a bounded subregion, with boundary $\partial \mm \in C^3$, of an $\tx{m}$-dimensional submanifold $Y$ of $\R^N$ such that $\overline{\mm} \subset \text{Int}(Y)$;
    \item\label{item:c} $\mm$ is a compact submanifold of $\R^N$ without boundary;
\end{enumerate}
with $\Omega \subset \R^n$ bounded and open, $n \geq 2$. Under suitable assumptions, the final outcome is the $C^{1,\beta}$-regularity for minima on an open set $\tilde{\Omega} \subset \Omega$ of full $n$-dimensional Lebesgue measure, for some $\beta \in (0,1)$, while the Hausdorff dimension of the complement $\Omega \setminus \tilde{\Omega}$ is at most equal to $n-p$. The case of $\Omega$ subset of an $n$-dimensional Riemannian manifold is also covered.

The classic double phase functional, corresponding to the choice of $\tx{g}(x,u,z) \equiv \tx{g}(x,z):= |z|^p + a(x)|z|^q$ with $0 \leq a(\cdot) \in C^{0,\alpha} $, $\alpha \in (0,1]$, was first introduced by Zhikov \cite{Z1, Z0} in the context of homogenization and in connection with Lavrentiev phenomenon. This type of energy functional plays a fundamental role in modeling of strongly anisotropic materials, with applications in elasticity and the theory of composites. 
The transition between \(p\)- and \(q\)-growth regions is governed by the vanishing of the coefficient \(a(\cdot)\). Over the last decade, a rich regularity theory for double phase and more general nonuniformly elliptic functionals has been established. In particular, the contributions of Baroni, Colombo, and Mingione \cite{BCM, comi2, comicz, comi1} developed a comprehensive regularity theory, which in turn generated a huge literature, see \cite{bb901,cjmpa, DLMM, G5, HO1,HO2,HO3,MNP,N,NC,jss} and references therein. For problems with linear and nearly linear growth see \cite{BG, BS, FPS, G1, G2, G3, G4}. The $(p,q)$-growth condition in \eqref{1f1} is directly tied to the nonuniform ellipticity of the associated Euler Lagrange system. Formally, a minimizer \(u\) of the double phase functional satisfies
\begin{equation}\label{el}
-\mathrm{div}\,A(x,Du) = 0,
\end{equation}
where the vector field 
\begin{equation*}
    A(x,z) = |z|^{p-2}z + (q/p) a(x)|z|^{q-2}z,
\end{equation*}
 and one sees that the (nonlocal) ellipticity ratio satisfies 
\begin{equation}\label{ratio}
\displaystyle \frac{\sup_{x\in \tx{B}} \text{highest eigenvalue of }\partial_z A(x, Du)} {\inf_{x\in \tx{B}} \text{lowest eigenvalue of }\partial_z A(x, Du)}\,\approx\, 1 + \nr{a}_{L^{\infty}(\tx{B})} |Du|^{q-p}
\end{equation}
on any ball $\tx{B} \subset \Omega$ touching the transition set \(\{ a(\cdot)=0\}\) and therefore becomes unbounded provided $a$ is not identically zero. For bounded minimizers, a condition of the type
\begin{equation*}
 1 < p < q < p + \alpha
\end{equation*}
precisely ensures that \(a(\cdot)\) vanishes fast enough to prevent uncontrollable degeneracy, guaranteeing maximal gradient regularity for local minimizers. The sharpness of this hypothesis was first observed by Esposito, Leonetti and  Mingione \cite{ELM}; see also \cite{balci2, balci4, ELM, FMM}. Historically, the notion of coupling growth exponents \(p\) and \(q\) to regulate the rate of blow up of the ellipticity ratio was pioneered by Marcellini \cite{ma2, ma4, ma1, ma5} in his foundational investigations of \((p,q)\)-nonuniformly elliptic integrals. Motivated by delicate phenomena in nonlinear elasticity, such as cavitation, the early investigations showed that choosing \( q \) sufficiently close to \( p \) yields a class of variational energies with controlled blow up. This perspective has also evolved into the framework of strongly nonuniformly elliptic functionals. In this setting, also the pointwise ellipticity ratio blows up (when $|z| \to +\infty$) in a nonbalanced fashion, namely
\begin{equation*}
\displaystyle \frac{\text{highest eigenvalue of }\partial_z A(x, Du)} {\text{lowest eigenvalue of }\partial_z A(x, Du)}\,\approx\, 1 + |z|^{\delta}, \quad \delta >0.
\end{equation*}
Although the underlying structure is strongly nonuniformly elliptic, recent advances have led to the development of a comprehensive Schauder theory \cite{DefS1, DefS3}, encompassing also problems at nearly linear growth, see \cite{DefS2, DFP, DFDFP}.\\

\noindent In our present study, we extend these insights to the vectorial obstacle setting.  Specifically, with a target $\mm$ as in \ref{item:a}, we consider  a  constrained local minimizer of the general double phase functional
\begin{align}\label{functionalG}
\mathcal{G}(u,\Omega)  &:=  \int_{\Omega} (g_{\alpha \beta}(x,u) B^{ij}(x, u) D_{\alpha}u^i D_{\beta}u^j\bigr)^{\frac{p}{2}} 
+
 a(x)\bigl(g_{\alpha \beta}(x,u) B^{ij}(x, u) D_{\alpha}u^i D_{\beta}u^j\bigr)^{\frac{q}{2}} \mathrm{d}x \notag \\ & :=  \int_{\Omega} \tx{g}(x,u,Du) \mathrm{d}x,
\end{align}
where $\Omega$  is a bounded open subset of $\R^n$, $n\geq 2$. The functions \(g\) and \(B\)  satisfy the following properties: for $\alpha,\beta=1,\dots,n$, $i,j=1,\dots,N$, functions $g_{\alpha\beta}: \overline{\Omega} \times \mathbb{R}^N  \to \mathbb{R}$, $B^{ij}: \overline{\Omega}\times \mathbb{R}^N \to \mathbb{R}$ are such that
\begin{equation} \label{ellipticity}
\begin{cases}
& g_{\alpha \beta} = g_{\beta \alpha}, \ B^{ij} = B^{ji} \\
    & g, B \in C^1 \\ 
    & \ell |\eta|^2 \leq g_{\alpha\beta}(x,u)\eta_\alpha \eta_\beta, \ \ell |\xi|^2 \leq  B^{ij}(x,u) \xi^i \xi^j ,
\end{cases}   
\end{equation}
for some positive constant $\ell$, for all $x \in \overline{\Omega}, \eta  \in \R^n, u,\xi \in \R^N$. 
While, in case \ref{item:b} or \ref{item:c}, we consider a constrained local minimizer of the classic double phase functional 
  \begin{equation}\label{functional}
\mathcal{H}(u,\Omega) := \int_{\Omega} |Du|^p + a(x) |Du|^q  \mathrm{d}\mathscr{H}^n:= \int_\Omega \ff(x,Du) \d\mathscr{H}^n,
\end{equation}
where $\Omega$ is a bounded open subset of an $n$-dimensional Riemannian manifold, \(n\ge 2\), $\mathscr{H}^n$ is the $n$-dimensional Hausdorff measure on $\Omega$ and 
\begin{equation}\label{pq}
0\leq a(\cdot) \in  C^{0,\alpha}(\Omega), \quad \alpha \in (0,1], \quad 1 < p < q < p + \alpha, \qquad q < N-1.
\end{equation}
We will always assume that the first $[q]$ homotopy groups of $\mm$ vanish identically, in other words
 \eqn{M}
 $$ \mm \text{ is } [q]\text{-connected},$$
 where $[q]$ is the integer part of $q$, see Definition \ref{j connected}.  This assumption plays a crucial role in our regularity analysis, particularly in ensuring the applicability of the finite energy extension Lemma~\ref{extensiontheorem}. For example, if $\mm$ is as in \eqref{M}, then  into category \ref{item:c} falls  the $(N-1)$-dimensional sphere in $\R^N$, i.e. \(\mathbb{S}^{N-1}\),  which is a compact manifold and \((N-2)\)-connected, hence it is admissible whenever \(q < N-1\).
  The following remark clarifies the range of admissible values for the exponent \( q \) that are compatible with this topological requirement.
\begin{remark}
\normalfont
The minimal connectivity assumption on the target \( \mm  \) required for the validity of the partial regularity result is that \( \mm \) is \(([q]-1)\)-connected. This condition is necessary and sufficient to ensure the validity of estimate \eqref{(3.5)Notes} in Lemma~\ref{extensiontheorem}. However, in order to apply Lemma~\ref{lem1} under the assumption of \(([q]-1)\)-connectedness, one must impose \( 1 \leq [q]-1 \leq N-2 \), which yields the restriction \( 2 \leq q < N \). Therefore, one may equivalently assume either \( 1 < q < N - 1 \) or \( 2 \leq q < N \), both of which are compatible with our framework.
\end{remark}
\noindent
 To precisely state our results, for $ \tilde{\mm}$ that is   
    $\overline{\mm}$ in case \ref{item:a},\ref{item:b} and $\mm$ in case \ref{item:c}, we introduce the admissible class
$$ W^{1,1}_{\loc}(\Omega, \tilde{\mm}) := \{u \in W^{1,1}_{\loc}(\Omega, \R^N):  \text{Im}(u) \subset \tilde{\mm} \text{ holds a.e.}\}$$
and, for convenience, we collect the main parameters of the problem in the quantity
$$ \tx{data} \equiv \tx{data}(n,N,\ell,L,p,q,\alpha,[a]_{0,\alpha},\tilde{\mm}),$$
where $L > \ell$ bounds $\|g_{\alpha\beta},B^{ij}\|_{L^\infty(\overline{\Omega}\times \tilde{\mm})}$ when we consider functional \eqref{functionalG} and the function $\|h_{\alpha\beta}\|_{L^\infty(\overline{\Omega})}$ arising in \eqref{funpre} when considering functional \eqref{functional}. Our main theorems assert that any constrained local minimizer \(u \in W^{1,1}_{\loc}\) is of class \(C^{1,\beta}\) on an open subset \(\tilde{\Omega} \subset \Omega\), called regular set, with \(\beta \in (0,1)\). Moreover, the complement \(\Omega \setminus \tilde{\Omega}\) is a closed set of zero Lebesgue measure. This is the so-called {\em partial regularity}: 
\begin{theorem} \label{maintheorem}
   Let $\Omega$ be a  bounded open subset of $\R^n$ and let $u \in W^{1,1}_{\loc}(\Omega, \overline{\mm})$ be a local minimizer of the functional $\mathcal{G}$ in \eqref{functionalG} under assumptions \eqref{ellipticity},\eqref{pq}, with $\mm$ as in \ref{item:a} and satisfying \eqref{M}. Then, there exists an open subset $\Tilde{\Omega} \subset \Omega$ such that
\eqn{stat}
    $$ Du \in C^{0,\beta}_{\loc}(\Tilde{\Omega}, \R^{N \times n}) \text{ and } |\Omega \setminus \Tilde{\Omega}| =0,$$
     for some $\beta \equiv \beta(\tx{data}) \in (0,1)$. 
\end{theorem}

\begin{theorem} \label{maintheorem2}
Let $\Omega$ be a  bounded open subset of an $n$-dimensional Riemannian manifold and let $u \in W^{1,1}_{\loc}(\Omega, \overline{\mm})$ be a local minimizer of the functional $\mathcal{H}$ in \eqref{functional} under assumptions \eqref{pq}, with $\mm$ as in \ref{item:b} and satisfying \eqref{M}. Then, there exists an open subset $\Tilde{\Omega} \subset \Omega$ such that \eqref{stat} holds true, for some $\beta \equiv \beta(\tx{data}) \in (0,1)$. 
\end{theorem}
\noindent 
In the case where \( \mm \) is a manifold as in \ref{item:c}, the result follows directly from the previous theorem, as the Euler-Lagrange equation simplifies due to the absence of points \( x \in \Omega \) such that \( u(x) \in \partial \mm \), see Theorem \ref{maintheorem3}.
In De Filippis and Mingione \cite{DFM}, a similar problem is treated for a density (also depending on \(u\)) which is bounded above and below by the double phase integrand, in the special case \(\mathsf{M} = \mathbb{S}^{n-1}\). 
When $a(\cdot) \equiv 0$, we are dealing with functionals with 
$p$-growth, as employed in the partial regularity results in \cite{KM1, KM, MMS}. Regularity results have also been developed in the setting where both minimizers and admissible maps take values in a manifold $\mm \subset \R^N$. The case $\mm = \mathbb{S}^{N-1}$ has been studied in depth, beginning with the classical works of Eells \& Sampson~\cite{ES} and Schoen \& Uhlenbeck~\cite{SU1, SU2} on harmonic maps, which correspond to the case $p = 2$ and $a(\cdot) \equiv 0$ in \eqref{functional}. For variational integrals with general \(p\)-growth (\(p \ge 2\)), Fuchs \cite{F1, F2, F3}, Hardt and Lin \cite{HL}, and Luckhaus \cite{L} have adapted harmonic map techniques to wider applications, see also \cite{HL2}. While the study of vectorial problems with variable exponent (i.e. $p(x)$-energies) has been treated in \cite{T} and extended to the manifold constrained setting in \cite{C2, DF}. In the quasiconvex case, analogous questions were addressed by Hopper \cite{H}. For free  boundary problems, we refer to the recent works \cite{FF,FFF} and references therein.\\

As mentioned before, these kinds of results find their place in a large literature devoted to so-called partial regularity, that is regularity of solutions outside a negligible closed subset. Such results, the only possible in the general vectorial case,  have been pioneered by Giusti, Miranda and Morrey in the classical papers \cite{giumi, morrey}, where partial H\"older continuity of solutions has been proved. Later on, more results of this kind have been proved at the boundary \cite{JM, gro} and \cite{fomi} for more general structures. We refer to the classical treatises \cite{gia, giusti} for an overview on gradient partial regularity as well as on the presentations given in \cite{KM1, KM}.\\

\noindent Now, we want to point out that for obstacle problems where the constraint set \( \mm \) is convex, the minimization property can be reformulated as a variational inequality. In contrast, for general \( \mm \) such linearization is nontrivial. Early work in \cite{F0} first tackled this issue, with subsequent refinements in \cite{D, DF2} developing sharper linearization techniques, here we adopt the version proposed in \cite{F1}. Our approach shows that Theorem \ref{maintheorem2} stems directly from Theorem \ref{maintheorem}, whose proof relies on two main ingredients. The first consists in adapting the classical framework developed by Fuchs \cite{F1}, which is based on a direct comparison between the constrained minimizer $u$ of the functional $\mathcal{G}$ in \eqref{functionalG} and the unconstrained minimizer (having $u$ as boundary datum) of $\mathcal{G}$ with frozen coefficients $g$ and $B$.  The second component of the approach builds on the harmonic approximation method used in \cite{DFM}. This allows us to involve in the comparison scheme also an harmonic map enjoying better regularity properties. In contrast to Fuchs’s work, a key role in this analysis is played by the topological assumption that \(\mathsf{M}\) is \([q]\)-connected. This property ensures the possibility of constructing finite energy extensions of Sobolev maps into \(\overline{\mm}\), which is formalized in Lemma~\ref{extensiontheorem},  this in turn yields an appropriate Caccioppoli's inequality. Notably, unlike in Fuchs’s setting, our approach also covers the case  $1<p<2$.
The proof of Theorem \ref{maintheorem} is based on the following partial regularity criterion: if there exists an $\varepsilon \in (0,1)$ such that on a ball $\tx{B}_{r}(x) \Subset \Omega$ it holds
\eqn{smallness1}
$$ \mint_{\br(x)} \ff(x,Du) \dx < \ff_{\br
(x)}^-\left ( \frac{\varepsilon}{r}\right ) $$
then $u$ is H\"older continuous
near $x$, where $\ff^-_{\tx{B}_r}$ is the frozen integrand defined in \eqref{frozen funct}. The $\varepsilon$-regularity condition \eqref{smallness1} provides an intrinsic quantified version of the amount of energy needed for regularity and it is related to the structure of the so-called singular set  $\tilde{\Sigma}:= \Omega \setminus \tilde{\Omega}$, that takes the form
$$ \tilde{\Sigma} :=\left \{ \tilde{x} \in \Omega: \limsup_{\rr \to 0} \left [\ff^{-}_{\tx{B}_{\rr}(\tilde{x})}\left ( \frac{1}{\rr}\right ) \right ]^{-1} \mint_{\tx{B}_\rr(\tilde{x})} \ff(x,Du) \dx >0\right \},$$
see Remark \ref{rm3.1}. To conclude, we underline that as an intermediate outcome (Step 1 of Theorem \ref{maintheorem}) we also get that
$$\text{for all  } \gamma \in (0,1) \text{  it holds  } u \in C^{0,\gamma}(\Tilde{\Omega}, \overline{\mm}) \text{  and  } |\Omega \setminus \Tilde{\Omega}| =0.$$
As a consequence of the partial regularity of the gradient established in Theorem \ref{maintheorem}, we are now in a position to derive  more refined estimates for the Hausdorff measure $\mathscr{H}$ of the singular set $\tilde{\Sigma}= \Omega \setminus \tilde{\Omega}$. We present this final improvement in the next results. For the complete notation we refer to Section \ref{sec4}. 
\begin{theorem} \label{theorem2}
    Let $\Omega$ be a  bounded open subset of $\R^n$ and let $u \in W^{1,1}_{\loc}(\Omega, \overline{\mm})$ be a local minimizer of the functional $\mathcal{G}$ in \eqref{functionalG}, assume \eqref{ellipticity}, \eqref{pq}, with $\mm$ as in \ref{item:a} and satisfying \eqref{M}. Then, there exists $\delta \equiv \delta(\tx{data})>0$ such that
    \eqn{suH}
    $$ \tx{H}(\cdot, Du) \in L^{1+\delta}_{\loc}(\Omega).$$
    Assume that $q(1+\delta) \leq n$, then
       \eqn{2.23}
    $$ \mathscr{H}_{\ff^{1+\delta}}(\tilde{\Sigma})=0.$$
    In particular, 
    \eqn{2.2} 
    $$ \mathscr{H}^{n-p-p\delta}(\tilde{\Sigma}) =0,$$
        \eqn{2.3} 
    $$  \mathscr{H}^{n-q-q\delta}(\tilde{\Sigma} \cap \{ a(x) >0\}) =0.$$
\end{theorem}
\noindent
Moreover, we also obtain the following bounds for the Hausdorff dimensions
    $$ \dim_{\mathscr{H}}(\tilde{\Sigma}) \leq n-p, \qquad \dim_{\mathscr{H}}(\tilde{\Sigma} \cap \{a(x)>0\}) \leq n-q,$$
    this is precisely the content of Corollary \ref{lcor}.
\begin{theorem} \label{the1.4}
      Let $\Omega$ be a  bounded open subset of an $n$-dimensional Riemannian manifold and let $u \in W^{1,1}_{\loc}(\Omega, \overline{\mm})$ be a local minimizer of the functional $\mathcal{H}$ in \eqref{functional}, assume \eqref{pq}, with $\mm$ as in \ref{item:b} and satisfying \eqref{M}. Then, there exists $\delta \equiv \delta(\tx{data})>0$ such that \eqref{suH} holds. Moreover, if $q(1+\delta) \leq n$, then \eqref{2.23}-\eqref{2.3} are true.
\end{theorem}
\noindent
The same conclusions apply when $u$ is a local minimizer of $\mathcal{H}$ in \eqref{functional} with values constrained to a manifold $\mm$ as specified in condition~\ref{item:c}.\\

\noindent The paper is organized as follows. In Section \ref{sec2}, we introduce the notation and collect all preliminary material needed for the main argument. This includes the relevant definitions and several technical lemmas that will be instrumental in the proof of the main result. Section \ref{sec3} contains the complete proof of the partial regularity theorem. In addition to deriving the regularity properties of the minimizer, we also describe the structure of the singular set, see Remark \ref{rm3.1}, and provide the bounds for the Hausdorff dimensions of the singular set. Finally, in Section \ref{sec4}, we improve the estimate on the Hausdorff measure of the singular set by building on the results established in the previous section.

\section{Preliminaries} \label{sec2}
\subsection{Notation}
\noindent
In this section, we introduce the notation used throughout the paper. Given $\tilde{x} \in \mathds{R}^n$ and $r > 0$, we define the open ball
\[
\tx{B}_r \equiv \tx{B}_r(\tilde{x}) := \{ x \in \mathds{R}^n : |x - \tilde{x}| < r \}.
\]
Unless otherwise specified, all balls considered will be centered at the same point. The symbol $c$ will denote a generic positive constant, not necessarily the same at each occurrence; its relevant dependencies will be explicitly indicated when needed. In the forthcoming estimates, any dependence of the constants on quantities related to the geometry of $\overline{\mm}$ (or $\mm$ in case \ref{item:c}) such as the $L^\infty$-norm of maps with range in $\overline{\mm}$  will be simply denoted by writing $c \equiv c(\overline{\mm})$. Note that the $L^\infty$-norm is finite due to the compactness of $\overline{\mm}$. If $U \subset \R^n$ is a measurable subset with bounded positive measure $0<|U|<\infty$ and $w$ is a measurable map, we set
$$(w)_U \equiv \mint_{U} w(x) \, {\rm d}x := \frac{1}{|U|} \int_U w(x) \, {\rm d}x.$$ 
As usual, with $\gamma \in (0,1)$, we denote
$$ [w]_{0,\gamma,U} := \sup_{x,y \in U, x \neq y}\frac{|w(x) - w(y)|}{|x-y|^\gamma}.$$
For the coefficient $a(\cdot)$ we abbreviate with $[a]_{0,\alpha} \equiv [a]_{0,\alpha,\Omega}$. 
We consider the auxiliary vector fields $V_p, V_q: \R^{N\times n}\to \R^{N\times n}$ defined by
\begin{equation} \label{V1}
    V_{t}(z) := |z|^{\frac{t-2}{2}}z, \quad t\in \{p,q\},
\end{equation}
whenever $z \in \R^{N\times n}$. To this extent, we collect some useful notations and properties, which will be needed in the proof of Theorem \ref{maintheorem}. For every $z_1, z_2 \in \R^{N \times n}$, we have
\begin{equation} \label{v}
    |V_{t}(z_{1})-V_{t}(z_{2})|\approx (|{z_{1}|}+|z_{2}|)^{(t-2)/2}|z_{1}-z_{2}|, \quad t\in \{p,q\},
\end{equation}
where the equivalence holds up to constants depending only on $n,N,t$, see \cite[Lemma 2.1]{Hamburger}. Two quantities that will play a crucial role throughout the paper are
\begin{align*}
    \mathcal{V}^{2}(x,z_1,z_2) := |V_{p}(z_1) - V_{p}(z_2)|^2 & + a(x)|V_{q}(z_1) - V_{q}(z_2)|^2,
\end{align*}
and
\begin{align*} 
        \mathcal{V}^{2}(z_1,z_2, {\tx{B}}_r) := |V_{p}(z_1) - V_{p}(z_2)|^2 & + \inf_{x \in \tx{B}_r}a(x)|V_{q}(z_1) - V_{q}(z_2)|^2,
\end{align*}
where ${\tx{B}}_r \subset \Omega$ is a ball. We report an important inequality (see \cite{BCM1,KM}) related to these fields. Let $\tilde{\tx{g}}$ be the function defined in \eqref{ng}. Then, for all $z_1,z_2 \in \R^{N\times n}$ and $x \in \Omega$, it follows
\eqn{V}
$$ \mathcal{V}^2(x,z_1,z_2) \leq c [\partial_z \tilde{\tx{g}}(x,z_1) -\partial_z \tilde{\tx{g}}(x,z_2)]\cdot (z_1 -z_2),$$
with $c\equiv c(n,N,\ell,p,q)$.
Let us now give the definitions of manifold $j$-connected.
\begin{definition}[Manifold $j$-connected]\label{j connected}
    Given an integer $j \geq 0$, a manifold $\tilde{\mm}$ is said to be $j$-connected if its
first $j$ homotopy groups vanish identically, that is $\pi_0(\tilde{\mm}) = \pi_1(\tilde{\mm}) = \dots= \pi_{j-1}(\tilde{\mm})=\pi_j(\tilde{\mm})= 0$.
\end{definition}
\noindent
The last definition we need is the one of a local minimizer of functional $\mathcal{H}$ as in \eqref{functional} with value in a compact target space $\tilde{\mm}$. For functional $\mathcal{G}$, with coefficients $g,B$ satisfying \eqref{ellipticity}, the definition does not change.
\begin{definition}[Constrained local minimizer]\label{defmin}
    A function $u \in W^{1,1}_{\loc}(\Omega, \tilde{\mm})$, is a local minimizer of the functional $\mathcal{H}$ in \eqref{functional}, if and only if $\tx{H}(\cdot, Du) \in L^1_{\rm{loc}}(\Omega)$
    and the minimality condition
    \begin{equation}\label{minimality condition}
       \mathcal{H}(u, \supp(u-v))\leq   \mathcal{H}(v, \supp(u-v))
    \end{equation}
    holds for all functions $v\in W^{1,1}_{\loc}(\Omega, \tilde{\mm})$ such that $\rm{supp}$$(u-v)\subset \Omega$. 
    \end{definition}
\noindent
 Before concluding this section, for $\tilde{\mm}$ compact submanifold, we introduce the spaces:
$$ W^{1,\tx{H}}(\Omega,\R^N):= \left \{ u \in W^{1,1}(\Omega,\R^N): \tx{H}(\cdot, u)+\tx{H}(\cdot, Du) \in L^{1}(\Omega) \right \},$$
$$ W^{1,\tx{H}}(\Omega,\tilde{\mm}):= \left \{ u \in W^{1,\tx{H}}(\Omega,\R^N):  \text{Im}(u) \subset \tilde{\mm} \text{ holds a.e.} \right \}.$$
\noindent
\subsection{Properties of the frozen double phase integrand} \label{properties}
We introduce the following Young function
\begin{equation}\label{frozen funct}
    \ff^-_{\tx{B}_{\rr}}(t) := t^p + \inf_{x \in \tx{B}_{\rr}} a(x)t^q,
\end{equation}
where $\tx{B}_{\rr} \Subset \Omega$ is a ball. We summarize some key properties of the functions 
$\ff^-_{\tx{B}_{\rr}}(t)$. First, observe that $\tx{H}^{-}_{\tx{B}_{\varrho}}(\cdot)$ is a Young function satisfying the $\Delta_2$-condition. Since $t \mapsto \tx{H}^{-}_{\tx{B}_{\rr}}(t)$ is strictly increasing and strictly convex, its inverse $(\tx{H}^{-}_{\tx{B}_{\varrho}})^{-1}$ is strictly increasing and strictly concave with $(\tx{H}^{-}_{\tx{B}_{\rr}})^{-1}(0) = 0$, making $(\tx{H}^{-}_{\tx{B}_{\rr}})^{-1}$ subadditive. Therefore, we observe that 
\eqn{lambdaf}
$$ (\ff^{-}_{\tx{B}_{\rr}})^{-1}(\lambda t) \leq (\lambda+1)(\ff^{-}_{\tx{B}_{\rr}})^{-1}(t), \quad \text{for all } \lambda \geq 0. $$
Furthermore, for any ball $\tx{B}_{\rr} \equiv \tx{B}_{\rr}(\tilde{x}) \Subset \Omega$, by the H\"older continuity of $a(\cdot)$, the function 
\eqn{contf}
$$(\tilde{x}, \varrho, t) \mapsto \tx{H}^{-}_{\tx{B}_{\varrho}(\tilde{x})}(t) \text{ is continuous on } \Omega \times [0, \infty) \times [0, \infty).$$ To conclude,  we observe that for a fixed $\tilde{x}$, if $\varrho_1 \leq \varrho_2$, it follows that 
\eqn{suiraggi}
$$\tx{H}^{-}_{\tx{B}_{\varrho_2}(\tilde{x})}(t) \leq \tx{H}^{-}_{\tx{B}_{\varrho_1}(\tilde{x})}(t)$$ holds uniformly in $t \geq 0$.
\subsection{Useful results}
In this section, we collect some auxiliary results used throughout the paper. We start with the following lemma that highlights good properties in terms of retractions for manifolds endowed with a relatively simple topology.
\begin{lemma} \label{lem1}
Let $\tilde{\mm} \subset \mathbb{R}^N$, $N \geq 3$, be a compact $l$-connected submanifold, for some integer $1 \leq l \leq N-2$, contained in an $N$-dimensional cube $Q$. Then there exists a closed $(N - l - 2)$-dimensional Lipschitz polyhedron $X \subset Q \setminus \tilde{\mm}$ and a locally Lipschitz retraction $\Psi : Q \setminus X \to \tilde{\mm}$ such that
\[
|D\Psi(x)| \leq \frac{c}{\mathrm{dist}(x, X)}, \quad \text{for all } x \in Q \setminus X,
\]
for some positive constant $c\equiv c(N,l,\tilde{\mm})$.
    
\end{lemma}
\begin{proof}
The result follows from \cite[Lemma 6.1]{HL}, which provides the original construction. Another proof based on Lipschitz extension properties for maps between Riemannian manifolds can be found in \cite[Lemma 4.5]{H}.

\end{proof}

\noindent A key difficulty in constrained variational problems lies in constructing admissible comparison maps, since standard convex combinations of a minimizer with suitable cut-off functions fail to preserve the constraint. To overcome this, we apply Lemma \ref{extensiontheorem}, which provides local energy control for a suitable projection $\tilde{w}$ of a competitor $w$. This is the content of the next lemma.
\begin{lemma}[Finite energy extension]\label{extensiontheorem}
Let $\hat{\Omega}\subset\Omega$ be an open bounded subset of $\mathbb{R}^n$ with Lipschitz boundary. Let $\tilde{\mm} \subset \mathbb{R}^N$, $N \geq 3$, be a compact $[q]$-connected submanifold, for $1<p<q<N-1$. Let $w \in (W^{1,\tx{H}} \cap L^{\infty})( \hat{\Omega}, \mathbb{R}^N)$ be such that $w(\partial \hat{\Omega})\subset \tilde{\mm}$, then there exists $\tilde{w} \in W^{1,\tx{H}}( \hat{\Omega}, \tilde{\mm})$ satisfying
\begin{equation}\label{LunaNotas}
\int_{ \hat{\Omega}} \tx{H}(x, D\tilde{w})\,\dd x \leq c \int_{ \hat{\Omega}} \tx{H}(x, Dw)\,\dd x,
\end{equation}
and $\tilde{w}-w\in W^{1,1}_0( \hat{\Omega}, \mathbb{R}^N)$, where $c\equiv c(n,N,p,q,\tilde{\mm})$.
\end{lemma}
\noindent
Before proceeding with the proof, let us introduce some notation. For any subset $A \subset \mathbb{R}^N$, we define \( \text{Unp}(A) \) as the set of all points \( x \in \mathbb{R}^N \) for which there exists a unique nearest point to $x$ in \( A \). For any \( t \in A \), let \( \text{reach}(A, t) \) denote the supremum of the set of all \( r > 0 \) such that the ball \( \{ x \in \mathbb{R}^N : |x - t| < r \} \) is entirely contained in \( \text{Unp}(A) \). We then define
\[
\text{reach}(A) := \inf_{t \in {A}}\text{reach}(A, t).
\]
We are now ready to prove Lemma \ref{extensiontheorem}, following the approach in \cite[Lemma 4.7]{H}.
\begin{proof}[Proof of lemma \ref{extensiontheorem}]
Let
$ V:= \{t \in \R^N: \text{dist}(t,\tilde{\mm}) < \sigma \}$, for some $0<\sigma < \text{reach}(\tilde{\mm})$, 
be a neighbourhood with the unique nearest point property.
Let \( \Pi : \overline{V} \to \tilde{\mm} \) denote the map that associates to each $t \in \overline{V}$ the unique $t_0 \in \tilde{\mm}$ such that $\text{dist}(t,\tilde{\mm})=|t-t_0|$. Then, by \cite[Theorem 2.4]{H}, function $\Pi$ is Lipschitz continuous. We observe that $\overline{V}$ and $\tilde{\mm}$ are homotopy equivalent spaces, since there exists a strong deformation retraction $h_\lambda(\xi)=(1-\lambda)\xi + \lambda\Pi(\xi)$, for $\xi \in \overline{V}$ and $0\leq \lambda\leq  1$. Therefore $\pi_j(\overline{V})=\pi_j(\tilde{M})=0$ for all $j =1,\dots,[q]$, see \cite[Proposition 1.17]{ha}. Now, since \( \tilde{\mm} \) is compact and \( w \) is bounded, there exists an \( N \)-dimensional cube \( Q \subset \mathbb{R}^N \) such that
\[
\tilde{\mm} \subset \overline{V} \subset Q \quad \text{and} \quad \text{dist}(w,\tilde{\mm}) \leq \frac{1}{2}\text{dist}(\tilde{\mm},\partial Q) \text{  almost everywhere}. 
\]
By Lemma \ref{lem1}, with \( l = [q] \), there exists a locally Lipschitz retraction
\[
\Psi : Q \setminus X \to \overline{V},
\]
where \( X \subset Q \setminus \overline{V} \) is a Lipschitz polyhedron of $(N - [q]-2)$-dimension, and is constructed to remain strictly away from \( \tilde{\mm} \). Define now the map \( P := \Pi \circ \Psi: Q \setminus X \to \tilde{\mm} \) then, by Lemma \ref{lem1}, it holds
\begin{equation}\label{(3.4)Notes}
|\nabla P(t)| \leq \frac{c}{\text{dist}(t, X)},
\end{equation}
with \( c \equiv c(N, \tilde{\mm}) \). By a change of variables, the definition of the dual skeleton, the fact that \( \dim(X) \leq N - [q] - 2 \) and $p,q < [q]+1$, we obtain
\begin{equation}\label{(3.5)Notes}
\int_Q \frac{\dt}{\text{dist}(t, X)^\tx{s}} < c, \ \text{for } \tx{s}={p,q},
\end{equation}
for some finite positive constant \( c = c(N, \tilde{\mm}, q) \). Now take \( 0 < \rr < \min\left\{ \frac{\sigma}{2}, \frac{\text{dist}(\tilde{\mm}, \partial Q)}{2} \right\} \), and for each \( b \in \tx{B}_\rr^N \) (the ball in $\R^N$ of radius $\rr$ and centered at $0$), define the translated sets
\[
Q_b := \{ t + b : t \in Q \}, \quad X_b := \{ t + b : t \in X \},
\]
as well as the translated retraction
\[
P_b : Q_b \setminus X_b \to \tilde{\mm}, \quad P_b(t) := P(t - b).
\]
Then, using the chain rule, Fubini’s theorem, and estimates \eqref{(3.4)Notes} and \eqref{(3.5)Notes}, we obtain
\begin{align}\label{(3.6)Notes}
\int_{\tx{B}_{\rr}^N} \int_{ \hat{\Omega}} \tx{H}(x, D(P_b(w)))\, \mathrm{d}x\, \mathrm{d}b 
&= \int_{\tx{B}_{\rr}^N} \int_{ \hat{\Omega}} \left[ |D(P_b(w))|^p + a(x)|D(P_b(w))|^q \right] \mathrm{d}x\, \mathrm{d}b \nonumber \\
&\leq \int_{ \hat{\Omega}} \left[ |Dw|^p \left( \int_{\tx{B}_\rr^N} |\nabla P(w - b)|^p \mathrm{d}b \right) + a(x) |Dw|^q \left( \int_{\tx{B}_\rr^N} |\nabla P(w - b)|^q \mathrm{d}b \right) \right] \mathrm{d}x \nonumber \\
&\leq \int_{ \hat{\Omega}} |Dw|^p \left( \int_Q |\nabla P(t)|^p \mathrm{d}t \right) \mathrm{d}x 
+ \int_{ \hat{\Omega}} a(x)|Dw|^q \left( \int_Q |\nabla P(t)|^q \mathrm{d}t \right) \mathrm{d}x \nonumber \\
&\leq c \int_{ \hat{\Omega}} \left( |Dw|^p + a(x) |Dw|^q \right) \mathrm{d}x 
= c \int_{ \hat{\Omega}} \tx{H}(x, Dw)\, \mathrm{d}x,
\end{align}
for some constant \( c = c(N, \tilde{\mm}, q) \). From \eqref{(3.6)Notes} and Markov’s inequality, there exists \( \tilde{b} \in \tx{B}_\rr^N \) and a constant \( c = c(n,N, \tilde{\mm}, q) \) such that
\begin{align}\label{(3.7)Notes}
\int_{ \hat{\Omega}} \tx{H}(x, D(P_{\tilde{b}}(w)))\, \mathrm{d}x 
\leq c \int_{ \hat{\Omega}} \tx{H}(x, Dw)\, \mathrm{d}x.
\end{align}

\noindent Since \( w(\partial  \hat{\Omega}) \subset \tilde{\mm} \), the map \( \tilde{w} := (P_{\tilde{b}}|_{\tilde{\mm}})^{-1}\circ P_{\tilde{b}} \circ w \) is well-defined. Moreover, since the inverse map \( (P_{\tilde{b}}|_{\tilde{\mm}})^{-1} \) is Lipschitz on \( \tilde{\mm} \), inequality~\eqref{(3.7)Notes} implies
\[
\int_{ \hat{\Omega}} \tx{H}(x, D\tilde{w})\, \mathrm{d}x \leq c \int_{ \hat{\Omega}} \tx{H}(x, Dw)\, \mathrm{d}x,
\]
with \( c = c(n,N, \tilde{\mm}, p, q) \), which shows that \eqref{LunaNotas} holds.  Finally, since \( w(\partial  \hat{\Omega}) \subset \tilde{\mm} \), it follows by construction that \( \tilde{w}|_{\partial  \hat{\Omega}} = w|_{\partial  \hat{\Omega}} \), completing the proof.

\end{proof}
\noindent
In the following lemma, we present a Caccioppoli-type inequality for a constrained local minimizer of the functional $\mathcal{G}$ in \eqref{functionalG}, with $\mm$ as in \ref{item:a}. The proof makes use of extension Lemma \ref{extensiontheorem}. Note that the results  still hold in the unconstrained case, i.e. $u \in (W^{1,\tx{H}}\cap L^{\infty}) (\Omega, \R^N)$. In this case, the constants appearing in the lemma will also depend on $\nr{u}_{L^{\infty}}$.
\begin{lemma}[Caccioppoli's inequality]\label{CaccioppoliLemma}
  Let $\Omega$ be a  bounded open subset of $\R^n$ and let $u \in W^{1,\tx{H}}(\Omega, \overline{\mm})$ be a constrained local minimizer of the functional $\mathcal{G}$ in \eqref{functionalG} under the assumptions $1<p<q \leq p + \alpha$, $q<N-1$, and \eqref{ellipticity}, with $\mm$ as in \ref{item:a} and  satisfying \eqref{M}. Then, there exists a constant $c \equiv c(n,N,\ell,L,p,q,\overline{\mm}) > 0$ such that for any pair of concentric balls $\tx{B}_r \subset \tx{B}_R \Subset \Omega$, it holds that
\begin{equation}\label{asteriscoNotes}
    \int_{\tx{B}_r} \tx{H}(x,Du) \, \dd x \leq c \int_{\tx{B}_R} \tx{H} \left( x, \frac{u - (u)_{\tx{B}_R}}{R - r} \right) \, \dd x.
\end{equation}
Moreover, if $R \leq 1$, then
\eqn{caccioppoli2}
$$
\int_{\tx{B}_{R/2}} \tx{H}(x,Du) \, dx \leq c \int_{\tx{B}_R} \tx{H}^-_{\tx{B}_R} \left( \frac{u - (u)_{\tx{B}_R}}{R} \right) \, \dx,
$$
for  $c \equiv c(n,N,\ell,L,p,q,[a]_{0,\alpha},\overline{\mm}) > 0$. Moreover, if
\begin{equation}\label{doubleasteriscNotes}
\inf_{\tx{B}_R}a(x) \leq 4[a]_{0,\alpha}R^\alpha
\end{equation}
holds and $R \leq 1$, then
\eqn{caccioppoli3}
$$\int_{\tx{B}_{R/2}} \tx{H}(x,Du) \, \dd x \leq c \int_{\tx{B}_R} \left| \frac{u-(u)_{\tx{B}_R}}{R}\right |^p \, \dd x,$$
with $c\equiv c(n,N,\ell,L,p,q,[a]_{0,\alpha},\overline{\mm})$.
\end{lemma}

\begin{proof}
Let \( r \leq t < s \leq R \). We take a cut-off function \( \eta \in C^\infty_c(\tx{B}_s) \) such that \( \chi_{\tx{B}_t} \leq \eta \leq \chi_{\tx{B}_s} \) and \( |D\eta| \leq 4(s - t)^{-1} \). Define
\[
w := u - \eta(u - (u)_{\tx{B}_R}).
\]
Since \( \eta \) is smooth and \( u \in W^{1,\tx{H}}(\tx{B}_s, \overline{\mm}) \), it follows that \( w \in W^{1,\tx{H}}(\tx{B}_s, \mathbb{R}^N) \) and \( u - w \in W^{1,1}_0(\tx{B}_s, \mathbb{R}^N) \).\\

\noindent
We observe that the choice \( \tilde{\mm} = \overline{\mm} \) in Lemma~\ref{extensiontheorem} is admissible, since the boundary \( \partial \mm \) is \( C^3 \)-regular. Then, there exists a function \( \tilde{w} \in W^{1,\tx{H}}(\tx{B}_s, \overline{\mm}) \), which is an admissible competitor for \( u \) in problem~\eqref{functional}, and satisfies~\eqref{LunaNotas} with \( \hat{\Omega} = \tx{B}_s \). Using the minimality of \( u \), $(\ref{ellipticity})_3$ and the fact that 
$ \|g_{\alpha\beta},B^{ij}\|_{L^\infty(\overline{\Omega}\times \overline{\mm})} \leq L$, we obtain:
\begin{align}\label{3.3DF}
    \int_{\tx{B}_s} \tx{H}(x, D u) \, \mathrm{d}x 
    &\leq c \int_{\tx{B}_s} \tx{g}(x,u, D u) \, \mathrm{d}x \nonumber \\
    &\leq c \int_{\tx{B}_s} \tx{g}(x,\tilde{w}, D \tilde{w}) \, \mathrm{d}x \nonumber \\
    &\leq c \int_{\tx{B}_s} \tx{H}(x, D \tilde{w}) \, \mathrm{d}x \nonumber \\
    &\leq c \int_{\tx{B}_s} \tx{H}(x, D u) \, \mathrm{d}x \nonumber \\
    &\leq c \left( \int_{\tx{B}_s \setminus \tx{B}_t} \tx{H}(x, D u) \, \mathrm{d}x + \int_{\tx{B}_s \setminus \tx{B}_t} \tx{H}\left(x, \frac{u - (u)_{\tx{B}_R}}{s - t}\right) \, \mathrm{d}x \right),
\end{align}
where \( c \equiv c(
N, \ell, L,p, q, \overline{\mm}) \).  
The proof of~\eqref{asteriscoNotes} is then completed by filling the hole and applying an iteration argument, see \cite{comi2}. \\

\noindent To establish~\eqref{caccioppoli2}, consider the estimate:
\begin{align*}
    \tx{H}\left(x, \frac{u - (u)_{\tx{B}_R}}{R}\right) 
    &= \left| \frac{u - (u)_{\tx{B}_R}}{R} \right|^p + a(x) \left| \frac{u - (u)_{\tx{B}_R}}{R} \right|^q \\
    &\leq \tx{H}^-_{\tx{B}_R}\left( \frac{u - (u)_{\tx{B}_R}}{R} \right) + \sup_{\tx{B}_R} \left[a(x) - \inf_{\tx{B}_R} a(x)\right] \left| \frac{u - (u)_{\tx{B}_R}}{R} \right|^q \\
    &\leq \tx{H}^-_{\tx{B}_R}\left( \frac{u - (u)_{\tx{B}_R}}{R} \right) + 2^{\alpha + q - p} [a]_{0,\alpha} R^{\alpha + p - q} \left| \frac{u - (u)_{\tx{B}_R}}{R} \right|^p \\
    &\leq c \, \tx{H}^-_{\tx{B}_R}\left( \frac{u - (u)_{\tx{B}_R}}{R} \right),
\end{align*}
with \( c \equiv c(p, q, \alpha,\overline{\mm}) \). Therefore,~\eqref{caccioppoli2} follows by applying~\eqref{asteriscoNotes} with \( r = R/2 \).\\

\noindent To prove~\eqref{caccioppoli3}, it suffices to use the estimate given by~\eqref{doubleasteriscNotes}:
\begin{align*}
    \inf_{x \in \tx{B}_R} a(x) \left| \frac{u - (u)_{\tx{B}_R}}{R} \right|^q 
    \leq 8 [a]_{0,\alpha} R^{\alpha + p - q} \left| \frac{u - (u)_{\tx{B}_R}}{R} \right|^p 
    \leq c \left| \frac{u - (u)_{\tx{B}_R}}{R} \right|^p,
\end{align*}
and then apply~\eqref{asteriscoNotes} once again.

\end{proof}
\noindent
We proceed with an intrinsic Sobolev-Poincaré inequality, the proof follows the one in \cite[Section 4]{comi2}, with minor modifications.
\begin{lemma}[Intrinsic Sobolev-Poincaré inequality]\label{IS-Pineq}
  Let $\Omega$ be a  bounded open subset of $\R^n$ and let $u \in (W^{1,\tx{H}} \cap L^{\infty})(\Omega, \mathbb{R}^N)$, $N \geq 1$, and let $\tx{B}_r \Subset \Omega$ be a ball with radius $r \leq 1$, assume that $q \leq p + \alpha$. Then, there exists $d \equiv d(n,p,q) < 1$ and $c \equiv c(n,N,p,q,[a]_{0,\alpha}, \|u\|_{L^{\infty}(\tx{B}_r)}) \geq 1$ such that:
\begin{equation}\label{sun}
\mint_{\tx{B}_r} \tx{H} \left( x, \frac{u - (u)_{\tx{B}_r}}{r} \right)\, \dd x \leq c \left( \mint_{\tx{B}_r} [\tx{H}(x,Du)]^d \, \dd x \right)^{1/d}.
\end{equation}
The quantity $u - (u)_{\tx{B}_r}$ can be replaced by $u$ if $\operatorname{tr}(u, \partial \tx{B}_r) \equiv 0$, or if $u \equiv 0$ on a subset $A \subset \tx{B}_r$ with $|A|/|\tx{B}_r| > \gamma > 0$, in which case $c$ also depends on $\gamma$.
\end{lemma}
\noindent
The following is an inner higher integrability lemma.
\begin{lemma}[Inner higher integrability]\label{lemmaHI}
  Let $\Omega$ be a  bounded open subset of $\R^n$ and let  $u \in W^{1,\tx{H}}(\Omega, \overline{\mm})$ be a constrained local minimizer of the functional $\mathcal{G}$ in \eqref{functionalG} under assumptions \eqref{pq}, \eqref{ellipticity}, with $\mm$ as in \ref{item:a} and satisfying \eqref{M}. Then, there exists $\displaystyle \delta \equiv \delta(\tx{data})$, positive integrability exponent, such that for all balls $\tx{B}_{2r} \subset \Omega$, with $r\leq 1$, it holds
\eqn{hi}
$$
\left( \mint_{\tx{B}_r} [\tx{H}(x,Du)]^{1+\delta}\, \dd x \right)^{1/(1+\delta)} \leq c \mint_{\tx{B}_{2r}} \tx{H}(x,Du) \, \dd x,
$$
with $c\equiv c(\tx{data})$.
\end{lemma}
\begin{proof}
Let $\tx{B}_{2r} \subset \Omega$. By the Caccioppoli inequality, Lemma \ref{CaccioppoliLemma}, there exists a constant $c \equiv c(n, N,\ell, L, p, q,\overline{\mm}) > 0$ such that
\[
\int_{\br} \tx{H}(x, Du)\, \dd x \leq c \int_{\brr} \tx{H}\left(x, \frac{u - (u)_{\brr}}{r} \right)\, \dd x.
\]
\noindent On the other hand, by the intrinsic Sobolev-Poincaré inequality, Lemma \ref{IS-Pineq}, we obtain:
\begin{equation}\label{squareNotes}
    \mint_{\br} \tx{H}\left(x, \frac{u - (u)_{\br}}{r} \right)\, \dd x \leq c \left( \mint_{\br} [\tx{H}(x, Du)]^d\, \dd x \right)^{1/d}
\end{equation}
where $d \equiv d(n, p, q) < 1$ and $c = c
(n, N, p, q, [a]_{0,\alpha}, \|u\|_{L^\infty(\Omega)})$.\\

\noindent Inequality \eqref{squareNotes} holds for every ball $\brr \Subset \Omega$. Therefore, we obtained a reverse  H\"older inequality that is uniform over all such balls.\\

\noindent By an application of Gehring's Lemma there exists $\delta \equiv \delta(\tx{data}) > 0$ and a constant
$
c \equiv c(\tx{data})>0
$
such that
\[
\left( \dashint_{\br} [\tx{H}(x, Du)]^{1 + \delta} \dd x \right)^{\frac{1}{1 + \delta}} \leq c \dashint_{\brr} \tx{H}(x, Du)\, \dd x,
\]
as wanted.
\end{proof}
\noindent
We establish the following approximation lemma, whose proof is omitted as it coincides with that in \cite[Section 4]{DF}. We consider a function $\tx{H}_0(t) := t^p + a_0 t^q, a_0 \geq 0$,
namely: no dependence on $x$ is considered.
\begin{lemma}[Harmonic approximation] \label{approximation}
Let $\tx{B}_r \subset \R^n$, with $r\leq 1$,  $\varepsilon \in (0,1]$ and let $v \in W^{1,\tx{H}_0}(\tx{B}_r, \mathbb{R}^N)$, $N\geq 1$, be a function such that 
\eqn{hiapprox}
$$ \left (\mint_{\tx{B}_{r/2}} [\tx{H}_0(Dv)]^{1+\delta} \dx \right )^{\frac{1}{1+\delta}} \leq c_1 \mint_{\tx{B}_r} \tx{H}_0(Dv) \dx,$$
for some $\delta \in (0,1)$, and
\eqn{secondaipotesi}
$$ \left | \mint_{\tx{B}_{r/2}} \partial_z \tx{H}_0(Dv) \cdot D\varphi \dx\right | \leq c_2 \varepsilon^t \mint_{\tx{B}_r}\left [ \tx{H}_0(Dv) + \tx{H}_0(\nr{D\varphi}_{L^{\infty}(\tx{B}_{r/2})})\right ]\dx,$$
where $t \in (0,1]$ and all $\varphi \in C^{\infty}_c(\tx{B}_{r/2}, \R^N)$, where $c_1, c_2$ are absolute constants. Then, there exists $\tilde{h} \in v+W^{1,\tx{H}_0}_0(\tx{B}_{r/2}, \R^N)$ such that 
$$ \int_{\tx{B}_{r/2}}\partial_z \tx{H}_0(D\tilde{h}) \cdot D\varphi \dx =0,$$
for all $\varphi \in W^{1,\tx{H}_0}_0(\tx{B}_{r/2}, \R^N)$. Moreover,
$$ \nr{\tilde{h}}_{L^{\infty}(\tx{B}_{r/2})} \leq \sqrt{N} \nr{v}_{L^{\infty}(\tx{B}_{r/2})}$$
and
$$ \mint_{\tx{B}_{r/2}} \left( |V_p(Dv)-V_p(D\tilde{h})|^2+a_0|V_q(Dv)-V_q(D\tilde{h})|^2\right ) \dx \leq c\varepsilon^m\mint_{\tx{B}_r}\tx{H}_0(Dv) \dx,$$
with $c \equiv c(n,N,p,q,c_1,c_2)$ and $m \equiv m(n,N,p,q,t,\delta)$. Finally, the function $\tilde{h}$ is the unique minimizer of 
$$ v+W^{1,\tx{H}_0}_0(\tx{B}_{r/2}, \R^N) \ni w \mapsto \int_{\tx{B}_{r/2}} \tx{H}_0(Dw) \dx.$$

\end{lemma}

\noindent
We now state a Morrey-type decay estimate in the vectorial case. The proof closely follows the argument developed in \cite{DFM}.
\begin{theorem}[Morrey type decay estimate]
  Let $\Omega$ be a  bounded open subset of $\R^n$ and let $v \in (W^{1,\tx{H}} \cap L^{\infty})(\Omega, \mathbb{R}^N)$ be a local minimizer of the functional $\mathcal{H}$ in \eqref{functional}
under assumption \eqref{pq}. Then, for every $\sigma \in (0,n]$, there exists a constant $c \equiv c(n,N,p,q,\alpha,[a]_{0,\alpha}, \|v\|_{L^{\infty}(\Omega)}, \sigma)$ such that
\eqn{morrey}
$$
\int_{\tx{B}_\varrho} \tx{H}(x,Dv) \, dx \leq c \left( \frac{\varrho}{r} \right)^{n - \sigma} \int_{\tx{B}_r} \tx{H}(x, Dv) \, dx
$$
for every pair of concentric balls $\tx{B}_\varrho \subset \tx{B}_r \subset \Omega$ with $\varrho \leq 1$.
\end{theorem}
\noindent
\section{Euler-Lagrange system}
\noindent
Now, following the construction outlined in \cite[Theorem 2.1]{F1}, we arrive at the following theorem that is a Euler-Lagrange system for constrained local minimizers. First we analyze case \ref{item:a}, with  $\Omega$ that is a  bounded open subset of $\R^n$, and introduce some notation. We observe that since $\partial \mm \in C^3$, the distance function $d(t):= \text{dist}(t, \partial \mm)$ is regular for $t \in \overline{\mm}$ near $\partial \mm$. By negative reflection we extend $d$ to a smooth function on a tubular neighborhood $U$ of $\partial \mm$ and define the vector fields
\eqn{vf}
$$ \nu(t):= \text{grad} d(t), \ v(x,t):= B^{-1}(x,t)(\nu(t)),$$
for $x \in \overline{\Omega}$ and $t  \in U$, $B^{-1}$ denotes the inverse of $(B^{ij})_{1 \leq i, j \leq N}$. Observe that, on  $\partial \mm$ the vector field $\nu$ is just the interior unit normal vector field. In the following we denote $\{ u \in \partial \mm \} = \{x \in \Omega: u(x) \in \partial \mm \}$.
\begin{theorem}
  Let $\Omega$ be a  bounded open subset of $\R^n$ and let  $u \in W^{1,1}_{\loc}
    (\Omega, \overline{\mm})$ be a local minimizer of the functional $\mathcal{G}$ in \eqref{functionalG}, under assumption  \eqref{ellipticity}, with $\mm$ as in \ref{item:a}. Then, for all $\varphi \in (W^{1,\tx{H}}_0\cap L^{\infty})(\Omega, \R^N)$ it holds
    \begin{align} \label{equation}
    & \int_{\Omega} \left ( pg_p(x,u,Du) + qa(x)g_q(x,u,Du) \right )\left ( G^{ij}_{\alpha \beta}(x,u)D_\alpha u^i D_\beta \varphi^j\right ) \dx \notag  \\ & \qquad + \int_{\Omega} \frac{1}{2} \left ( pg_p(x,u,Du) + qa(x)g_q(x,u,Du)\right )\left ( D_{u^l}G^{ij}_{\alpha \beta} (x,u)D_\alpha u^i D_\beta u^j \varphi^l\right )   \dx \notag \\ 
& \quad = \int_{\{ u \in \partial \mm \}} \theta \frac{\nu(u)\varphi}{\nu(u) \cdot v(x,u)} \left ( pg_p(x,u,Du) + qa(x)g_q(x,u,Du)\right )\cdot \left ( G^{ij}_{\alpha \beta}(x,u)D_\alpha u^i D_\beta (v(x,u)^j) \right . \notag \\ & \qquad \left .+ \frac{1}{2}D_{u^l} G^{ij}_{\alpha \beta} (x,u) D_\alpha u^i D_\beta u^j v(x,u)^l\right )   \dx,
    \end{align}
    where $\theta: \Omega \to [0,1]$ is a Lebesgue measurable density function, and
    $$ G^{ij}_{\alpha \beta}(x,u) := g_{\alpha \beta}(x,u) B^{ij}(x,u), \ g_\tx{s} (x,u,z) := (G^{ij}_{\alpha \beta}(x,u)z_{\alpha}^i z_{\beta}^j)^{\frac{\tx{s}}{2}-1}, \ \text{for } \tx{s}=p,q.$$
\end{theorem}
\begin{proof}
We start covering $\partial \mm$ with balls $\tx{B}_k = \tx{B}_r(y_k)$, $y_k \in \partial \mm$, $k =1,\dots,\tx{L}$ such that $\tx{B}_{3r}(y_k) \Subset U$. Then, we choose a partition of the unity $\{ \varphi_k\}$ such that
$$ \text{supp}(\varphi_k) \Subset \tx{B}_{2r}(y_k), \quad \sum_{k=1}^\tx{L} \varphi_k =1 \text{ on } \bigcup_{k=1}^\tx{L} \tx{B}_k \supset \partial \mm.$$
Now, recall that from  $(\ref{ellipticity})_{3}$, for $x \in \overline{\Omega}$ and $t  \in U$, we get
$$ \nu(t)\cdot v(x,t) >0$$ and consider the vector fields $T_{k,1}, \dots, T_{k,N-1}$ with the following properties
$$ \text{supp}(T_{k,i}) \Subset \tx{B}_{3r}(y_k), \quad T_{k,i} \cdot T_{k,j} = \delta_{i,j}, \quad T_{k,i} \cdot \nu = 0 \text{ on } \tx{B}_{2r}(y_k), \qquad \text{for } i=1,\dots, N-1.$$
Then, locally we can decompose any test vector $\psi \in C^{1}_0(\Omega, \R^N)$ into the following sum
\eqn{eta}
$$ \varphi_k (u) \psi = \eta_k v(x,u) + \sum_{i=1}^{N-1} \eta_{k,i} T_{k,i} (u)$$
where
$$ \eta_k = \varphi_k (u) \psi \cdot \nu(u)/(\nu(u)\cdot v(\cdot,u)), \quad \eta_{k,i} = \varphi_k(u) T_{k,i}(u) \psi - \eta_k v(\cdot,u) \cdot T_{k,i}(u),$$
note that $\eta_k, \eta_{k,i} \in (W^{1,\tx{H}} \cap L^{\infty})(\Omega)$ have compact support in $\Omega$. Now, consider a smooth function $h_\varepsilon: [0,\infty) \to [0,1]$  so that $h_\varepsilon (s)=1$ for $0\leq s \leq \varepsilon$, $h_\varepsilon(s)=0$ for $s \geq 2\varepsilon$ and $h_\varepsilon' \leq 0$ and any $\eta \in C^{1}_0(\Omega)$ such that $\eta \geq 0$.

Let us first study the tangential contribution in \eqref{eta}. We take a generic vector field $ T : \R^N \to \R^N$ with support contained in a ball centered in $\partial \mm$ and such that $T \cdot \nu =0$, denoting with $\Phi(s,u)$ the flow of $T$, we can choose $$v_\delta := \Phi(\delta\eta h_\varepsilon(d(u)),u)$$ as admissible test function, for  $|\delta|$ small. Then, from the minimality of $u$, we get
\begin{align} \label{eq2}
    & \int_\Omega (pg_p(x,u,Du) + qa(x)g_q(x,u,Du))G^{ij}_{\alpha  \beta}(x,u)D_\alpha  u^i D_\beta (\eta h_\varepsilon (d(u))T^j(u)) \dx \notag \\ & \qquad + \int_{\Omega} \frac{1}{2} (pg_p(x,u,Du) + qa(x)g_q(x,u,Du)) D_{u^l} G^{ij}_{\alpha \beta} (x,u) D_\alpha u^i D_\beta u^j \eta h_\varepsilon (d(u)) T^l(u) \dx =0.
\end{align}
As regards the nontangential contribution in \eqref{eta}, we observe that for a small positive $\delta$ the following test function  is admissible
$$ u_\delta := u + \delta\eta h_\varepsilon(d(u))v(\cdot,u).$$
Moreover, by Riesz representation theorem there exists a Radon measure $\lambda  \geq 0$ on $\Omega$ such that
\begin{align} \label{eq1}
    \int_\Omega  \eta \d\lambda &= \int_{\Omega} (p g_p(x,u,Du)+ qa(x) g_q(x,u,Du)) G^{ij}_{\alpha \beta}(x,u)D_{\alpha}u^i D_\beta(\eta h_\varepsilon(d(u)))v(x,u)^j) \dx  \notag \\ & \quad + \int_{\Omega}\frac{1}{2}(p g_p(x,u,Du)+ qa(x) g_q(x,u,Du))D_{u^l} G^{ij}_{\alpha \beta} (x,u)D_\alpha u^i D_\beta u^j \eta h_\varepsilon (d(u)) v(x,u)^l \dx =: \tx{I} + \tx{II}.
\end{align}
Note that, for $\tilde{\varepsilon} \neq \varepsilon$ and $|\delta| \ll 1$, the variation $u+\delta\eta(h_\varepsilon (d(u))- h_{\tilde{\varepsilon}}(d(u)))v(\cdot,u)$ is  admissible,  so $\lambda$  is independent of $\varepsilon$. Now, we  observe that since $\nu(u)^i D_\alpha u^i = D_\alpha (d(u)) =0$ on $\{ u \in \partial \mm \}$ and $h_\varepsilon '(d(u)) \leq 0$, we get
$$ \tx{I} \overset{\varepsilon \to 0}{\longrightarrow} \int_{\{ u \in \partial \mm \}} (pg_p(x,u,Du)+qa(x)g_q(x,u,Du)) G^{ij}_{\alpha \beta}(x,u) D_\alpha u^i D_\beta (v(x,u)^j) \eta \dx$$
and
$$ \tx{II} \overset{\varepsilon \to 0}{\longrightarrow}  \int_{\{ u \in \partial \mm\} } \frac{1}{2}(pg_p(x,u,Du) + qa(x)g_q(x,u,Du)) D_{u^l}G^{ij}_{\alpha \beta}(x,u) D_\alpha u^i D_\beta u^j v(x,u)^l \eta \dx.$$
Therefore, Radon Nikodym theorem gives the existence of a density function $\theta: \Omega \to [0,1]$ such that $\lambda$ can be written as
\begin{align} \label{lambda}
    \lambda & =  \theta (pg_p(x,u,Du)+qa(x)g_q(x,u,Du)) \big \{ G^{ij}_{\alpha \beta} (x,u) D_\alpha u^i D_\beta (v(x,u)^j) \notag \\ & \qquad + \frac{1}{2}D_{u^l}G^{ij}_{\alpha \beta}(x,u)D_\alpha u^i D_\beta u^j v(x,u)^l \big \} \mathscr{L}^n \lfloor \{ u \in \partial \mm\},
\end{align}
where $\mathscr{L}^n$ is the $n$-dimensional Lebesgue measure, this shows that we can take $\eta \in (W^{1,\tx{H}}_0 \cap L^\infty)(\Omega)$.

Now, we apply \eqref{eq2} with
\[
T = T_{k,i}, \quad \eta = \eta_{k,i}, \quad \text{for } k = 1, \dots, \tx{L},
\]
and then sum over \(i\). Next, we apply \eqref{eq1} with \(\eta = \eta_k\) for \(k = 1, \dots, \tx{L}\). Combining these, we obtain
\begin{align*}
  \int_{\Omega} \frac{\psi \cdot \nu(u)}{\nu(u) \cdot v(x,u)} \varphi_k(u) \d\lambda    & =  \int_\Omega (pg_p(x,u,Du) + qa(x) g_q(x,u,Du)) G^{ij}_{\alpha \beta} (x,u) \\ & \qquad \cdot  D_\alpha u^i D_\beta \left [\left ( \sum_{l=1}^{N-1} \eta_{k,l}T_{k,l}(u) + \eta_k v(x,u) \right )^j h_\varepsilon (d(u)) \right ] \dx \\ & \quad + \int_{\Omega} \frac{1}{2}(pg_p(x,u,Du) + qa(x) g_q(x,u,Du))D_{u^h}G^{ij}_{\alpha \beta}(x,u) \\ & \qquad \cdot D_\alpha u^i D_\beta u^j \left ( \sum_{l=1}^{N-1} \eta_{k,l} T_{k,l} + \eta_k v(x,u)\right )^h h_\varepsilon (d(u)) \dx.
\end{align*}
Using \eqref{eta}, the equation above turns into
\begin{align*}
   & \int_{\Omega} \frac{\psi \cdot \nu(u)}{\nu(u) \cdot v(x,u)} \varphi_k(u) \d\lambda \\ & \quad = \int_\Omega (pg_p(x,u,Du) + qa(x)g_q(x,u,Du))G^{ij}_{\alpha \beta} (x,u) D_\alpha u^i D_\beta \left ( \varphi_k(u) h_\varepsilon (d(u)) \psi^j\right ) \dx \\ & \qquad + \int_{\Omega} \frac{1}{2}(pg_p(x,u,Du) + qa(x)g_q(x,u,Du))D_{u^h}G^{ij}_{\alpha \beta}(x,u) D_\alpha u^i D_\beta u^j \varphi_k(u) h_\varepsilon (d(u))  \psi^h \dx.
\end{align*}
Now, from the last displayed equation, taking the sum with respect to $k$, observing that $\sum_{k=1}^{\tx{L}} \varphi_k =1$ on the set where $h_\varepsilon \circ d \neq 0$, using \eqref{lambda} and adding the identity
$$ \frac{d}{d\delta|0}\mathcal{G}(w_\delta, \Omega) =0, \quad w_\delta := u + \delta(1-h_\varepsilon(d(u))) \psi,$$
we finally get \eqref{equation}.
\end{proof}
\noindent
Let us now turn to case \ref{item:b}. Recall that $u$ is a local minimizer of $\mathcal{H}$ in \eqref{functional}. Let $d_Y$ denote the Riemannian distance of points $y,t$ in $Y$ and $\rho(t):=d_Y(t,\partial \mm)$ for $t \in \overline{\mm}$ near $\partial \mm$. By negative reflection we extend $\rho$ to a smooth function in a tubular neighborhood  $V_\mm$ of $\partial \mm$ in $Y$. Define $\nu(t):=  \text{grad}_Y\rho(t)$. Then, denoting with $du^i$ the gradient of $u^i$ with respect  to the metric  on $\Omega$ and following the same calculations of before, we get
$$ \int_\Omega (p|du|^{p-2}+qa(x)|du|^{q-2})du^i \cdot d(\eta h_\varepsilon (\rho(u))\nu(u)^i) \text{d}\mathscr{H}^n= \int_{\Omega} \eta \d\lambda,$$
where $\lambda \geq 0$ is a Radon measure on $\Omega$ of the form
$$ \lambda = \theta\left \{p|du|^{p-2} +qa(x)|du|^{q-2} \right \}du^i \cdot d(\nu(u)^i) \mathscr{H}^n \lfloor \{ u \in \partial \mm\},$$
where $\theta: \Omega \to [0,1]$ is a $\mathscr{H}^n$-measurable density function. In similar fashion to before, we obtain
\eqn{2.8}
$$ \int_{\Omega} (p|du|^{p-2}+qa(x)|du|^{q-2})du^i \cdot d\varphi^i \d \mathscr{H}^n= \int_{\Omega} \varphi \cdot \nu(u) \d\lambda,$$
for vector fields $\varphi \in (W^{1,\tx{H}}_0\cap L^{\infty})(\Omega, \R^N)$ along $u$. If $\Pi: V_Y \to Y$ is the smooth nearest point retraction defined on a suitable tubular neighborhood $V_Y$ of $Y$, and if $\psi \in C^{1}_0(\Omega, \R^N)$ is arbitrary, then $\varphi:= D\Pi|_u(\psi)$ is a field along $u$ for which \eqref{2.8} holds. So, with respect to \eqref{equation}, the left-hand side now involves a dependence on $u$ through the projection. Therefore, the following theorem can be readily shown.
\begin{theorem}
  Let $\Omega$ be a  bounded open subset of an $n$-dimensional Riemannian manifold and let $u \in W^{1,1}_{\loc}
    (\Omega, \overline{\mm})$ be a local minimizer of the functional $\mathcal{H}$ in \eqref{functional}, with $\mm$ as in \ref{item:b}. Then, for all $\varphi \in (W^{1,\tx{H}}_0\cap L^{\infty})(\Omega, \R^N)$ it holds
\begin{align} \label{euler2}
   & \int_{\Omega}(p|du|^{p-2} +qa(x)|du|^{q-2})du^i \cdot d\varphi^i\d\mathscr{H}^n + \int_{\Omega}(p|du|^{p-2} +qa(x)|du|^{q-2})  du^i \cdot du^kD^2_{lk}\Pi^i(u) \varphi^l \d\mathscr{H}^n \notag \\ & \qquad = \int_{\{u\in \partial \mm\}}\theta \varphi\cdot \nu (p|du|^{p-2}+qa(x)|du|^{q-2})du^i \cdot d(\nu(u)^i) \d\mathscr{H}^n,
\end{align}
where $\theta: \Omega \to [0,1]$ is a $\mathscr{H}^n$-measurable density function.
\end{theorem}
\noindent
Note that, for a submanifold as in \ref{item:c}, the Euler Lagrange system is the one in \eqref{euler2} without the boundary term.

\section{Partial regularity result} \label{sec3}
\noindent
We start this section proving the gradient partial regularity in the case of Theorem \ref{maintheorem}. The proof is organized in three steps.  In the first two steps, we show that the minimizer is locally  $\gamma$-H\"older continuous, for any exponent \(\gamma\in (0,1)\), outside a set of Lebesgue measure zero.  In the third step, we establish that the gradient is locally  $\beta$-H\"older continuous for some exponent \(\beta\in(0,1)\).  We also provide a precise description of the regular and singular sets, as well as the upper bounds for the Hausdorff dimension of the singular set.
\begin{proof}[Proof of Theorem \ref{maintheorem}]
Let $u \in W^{1,\tx{H}}_{\loc}(\Omega, \overline{\mm})$ be as in the statement of Theorem \ref{maintheorem}. We begin by considering the case \( p(1+\delta) \leq n \), where \(\delta\) denotes the exponent appearing in Lemma~\ref{lemmaHI}.
We divide the proof into three steps. \\
\textbf{Step 1:} H\"older regularity of the minimizer.  \\
Let $\bb_{r} \equiv \bb_{r}(\tilde{x})$ be any ball such that $\bb_{2r} \Subset \Omega$, where $r \in (0,r_*/2)$, with $r_* \in (0,1)$ that will be chosen later. Assume
\eqn{smallness2}
$$ \mint_{\brr} \ff(x,Du) \dx < \ff_{\brr}^-\left ( \frac{\varepsilon}{2r}\right ),$$
for some $\varepsilon \in (0,1)$ that will be defined later. Let $v \in u+ W^{1,\tx{H}}_0(\bb_r, \R^N)$ be the solution to the Dirichlet problem
\eqn{dirichlet}
$$  u+ W^{1,\tx{H}}_0(\bb_r, \R^N) \ni w \mapsto \min_{w \in u+ W^{1,\tx{H}}_0(\bb_r, \R^N)}\int_{\bb_r} \tilde{\g}(x,Dw) \dx,$$
where 
\begin{align} \label{ng}
   \tilde{\g}(x,z) & :=(g_{\alpha \beta}(\Tilde{x},(u)_{\tx{B}_{r}})B^{ij}(\Tilde{x},(u)_{\tx{B}_{r}})z^i_{\alpha}z^j_{\beta})^{\frac{p}{2}} + a(x)(g_{\alpha \beta}(\Tilde{x},(u)_{\tx{B}_{r}})B^{ij}(\Tilde{x},(u)_{\tx{B}_{r}})z^i_{\alpha}z^j_{\beta})^{\frac{q}{2}} \notag \\ & := (\tilde{g}_{\alpha \beta}\tilde{B}^{ij}z^i_{\alpha}z^j_{\beta})^{\frac{p}{2}} + a(x)(\tilde{g}_{\alpha \beta}\tilde{B}^{ij}z^i_{\alpha}z^j_{\beta})^{\frac{q}{2}}.
\end{align}
We observe that for all $x \in \Omega$, $z \in \R^{N\times n}$, it holds
\eqn{growth}
$$ \ell \ff(x,z) \leq \tilde{\g}(x,z) \leq L \ff(x,z).$$
Now, we can rewrite equation \eqref{equation} in terms of the derivative with respect to $z$ of the integrand $\tx{g}$ in \eqref{functionalG}, and a term $\tx{f}$ that arises from the $C^1$-dependence of $\tx{g}$ on the second variable $u$ as well as the constraint $\overline{\mm}$.  Then, for all $\varphi \in  (W^{1,\tx{H}}_0 \cap L^{\infty}  )(\Omega, \R^N)$, it holds
\eqn{equation2}
$$ \int_{\Omega} \partial_z\g(x,u,Du) \cdot D \varphi \dx = \int_{\Omega} \tx{f}(x,u,Du) \cdot \varphi \dx. $$
Note that $\tx{f}: \Omega \times \R^N \times \R^{N\times n} \to \R^N$ satisfies the growth condition
$$ |\tx{f}(x,u,z)| \leq c(|z|^p +a(x)|z|^q + |z|^{p-1} + a(x)|z|^{q-1}),$$
for some positive constant $c \equiv c(n,N,L,p,q,\nr{D_u B}_{L^{\infty}(\overline{\Omega}\times \overline{\mm})},\nr{D_u g}_{L^{\infty}(\overline{\Omega}\times \overline{\mm})},\overline{\mm})$. Now, since matrices $\tilde{g}, \tilde{B}$ in \eqref{ng} are constant, symmetric, and uniformly elliptic, we can transform integral $\int \tilde{\tx{g}}(x,Dw)$ into the classic double phase one, by means of a change of coordinates both in the domain of definition and in the image space, see \eqref{cv}. Then, we deduce that since $u$ is bounded also $v \in L^{\infty}(\br, \R^N)$. Now, by \eqref{dirichlet}, $v$ solves
\eqn{equation3}
$$ \int_{\br}\partial_z \tilde{\g}(x,Dv) \cdot D\varphi \dx =0,$$
for all $\varphi \in (W^{1,\tx{H}}_0\cap L^{\infty})(\br, \R^N)$. Then, using \eqref{V} and \eqref{equation2}, \eqref{equation3} with the choice of $\varphi := u-v$, we get
\begin{align} \label{3eq}
    \mint_{\br} \mathcal{V}^2(x,Du,Dv) \dx & \leq c\mint_{\br}|\partial_z \tilde{\g}(x,Du) - \partial_z \g(x,u,Du)||Du-Dv| \dx \notag \\ & \quad + c\mint_{\br} (|Du|^{p-1} + a(x)|Du|^{q-1})|u-v| \dx \notag \\ & \quad + c\mint_{\br} (|Du|^{p} + a(x)|Du|^{q})|u-v| \dx =: \tx{I} + \tx{II}+ \tx{III},
\end{align}
with $c \equiv c(n,N,\ell,L,p,q,\nr{D_u B}_{L^{\infty}(\overline{\Omega}\times \overline{\mm})},\nr{D_u g}_{L^{\infty}(\overline{\Omega}\times \overline{\mm})},\overline{\mm})$.
Let us estimate the three terms in the right hand side above. First, we observe that by Section \ref{properties}, Jensen's inequality, Sobolev-Poincar\'e inequality for double phase functional (i.e. \eqref{sun} with $d=1$) and by the minimality of $v$ in conjunction with \eqref{growth}, we have
\begin{align} 
    \mint_{\br} |u-v| \dx & \mathrel{\eqmathbox{\overset{\mathrm{}}{\leq}}}  r \left (\ff^{-}_{\brr} \right )^{-1} \circ \left (\ff^{-}_{\brr} \right )\left (\mint_{\br} \left | \frac{u-v}{r} \right |\dx \right )
    \notag \\ & \mathrel{\eqmathbox{\overset{\mathrm{}}{\leq}}}  r \left (\ff^{-}_{\brr} \right )^{-1}\left ( \mint_{\br} \ff^-_{\brr}\left (\frac{u-v}{r} \right )\dx \right )  \notag  \\ & \mathrel{\eqmathbox{\overset{\mathrm{}}{\leq}}}  r \left (\ff^{-}_{\brr} \right )^{-1}\left ( \mint_{\br} \ff\left (x, \frac{u-v}{r} \right )\dx \right )  \notag\\ & \mathrel{\eqmathbox{\overset{\mathrm{}}{\leq}}} cr \left (\ff^{-}_{\brr} \right )^{-1} \left (\mint_{\br} \ff\left (x, Du-Dv \right )\dx \right )\notag\\ & \mathrel{\eqmathbox{\overset{\mathrm{}}{\leq}}} cr \left (\ff^{-}_{\brr} \right )^{-1} \left (\mint_{\brr} \ff\left (x, Du \right )\dx \right ) \label{uv} \\ & \mathrel{\eqmathbox{\overset{\mathrm{\eqref{smallness2}}}{\leq}}} cr \left (\ff^{-}_{\brr} \right )^{-1} \circ \ff^-_{\tx{B}_{2r}(\tilde{x})}\left (\frac{\varepsilon}{2r}\right ) \leq c\varepsilon \label{uv2},
\end{align}
with $c \equiv c(n,N,\ell,L,p,q,\alpha,[a]_{0,\alpha})$. Note that, in a similar fashion, we get
\begin{align} \label{uu}
    \mint_{\br} |(u)_{\br}-u| \dx  \leq cr \left (\ff^{-}_{\brr} \right )^{-1} \left (\mint_{\brr} \ff\left (x, Du \right )\dx \right ). 
\end{align}
% Moreover, observe that by \eqref{1f}, \eqref{lambdaf} and \eqref{smallness2}, we have
% \begin{align} \label{rstima}
% & r \mint_{\brr}\ff(x,Du) \dx + \left [ r\left (\tx{H}^-_{\brr}\right )^{-1}\left ( \mint_{\brr} \ff(x,Du) \dx\right )\right ]^{\frac{\delta}{1+\delta}} \mint_{\brr}\ff(x,Du) \dx \notag \\ & \qquad \leq c\left [  r^{\frac{1+\delta}{\delta}} +r\left (\tx{H}^-_{\brr}\right )^{-1}\left ( \mint_{\brr} \ff(x,Du) \dx\right )\right ]^{\frac{\delta}{1+\delta}}\mint_{\brr}\ff(x,Du) \dx \notag \\ & \qquad \leq c\left [  r\left ( 1+\left (\tx{H}^-_{\brr}\right )^{-1}\left ( \mint_{\brr} \ff(x,Du) \dx\right )\right )\right ]^{\frac{\delta}{1+\delta}}\mint_{\brr}\ff(x,Du) \dx \notag \\ & \qquad \leq c\left [  r \left (\tx{H}^-_{\brr}\right )^{-1}\left ( \mint_{\brr} (1+\ff(x,Du) )\dx\right )\right ]^{\frac{\delta}{1+\delta}}\mint_{\brr}\ff(x,Du) \dx \notag \\ & \qquad \leq  c\left [r \left (
%    \ff^{-}_{\brr} \right )^{-1} \circ \ff^{-}_{\brr}\left (\frac{\varepsilon}{2r} \right )\right ]^{\frac{\delta}{1+\delta}} \mint_{\brr}\ff(x,Du) \dx\notag\\ & \qquad \leq c \varepsilon^{\frac{\delta}{1+\delta}}\mint_{\brr}\ff(x,Du) \dx,
% \end{align}
% with $c\equiv c(n,N,\ell,L,p,q,\alpha,[a]_{0,\alpha}, \nr{a}_{L^{\infty}})$. 
\noindent
For $\tx{III}$, using H\"older inequality with conjugate exponents $(1+\delta, (1+\delta)/\delta)$, \eqref{hi}, \eqref{uv}, we obtain
\begin{align} \label{e1}
   \tx{III} & \leq  c\left ( \mint_{\br} \ff(x,Du)^{1+\delta} \dx \right )^{\frac{1}{1+\delta}} \left (  \mint_{\br} |u-v|^{\frac{1+\delta}{\delta}} \dx\right )^{\frac{\delta}{1+\delta}} \notag\\ & \leq c \mint_{\brr} \ff(x,Du) \dx \left (  \mint_{\br} |u-v| \dx\right )^{\frac{\delta}{1+\delta}} \notag\\ & \leq c\mint_{\brr}\ff(x,Du) \dx  \left [ r\left (\tx{H}^-_{\brr}\right )^{-1}\left ( \mint_{\brr} \ff(x,Du) \dx\right )\right ]^{\frac{\delta}{1+\delta}},
\end{align}
with $$c \equiv c(\tx{data}_1) :=c(\tx{data}, \nr{D_ug}_{L^{\infty}},\nr{D_uB}_{L^{\infty}} ).$$ For $\tx{II}$ we apply H\"older's inequality, Sobolev-Poincarè inequality for double phase functional (with $d=1$) and the minimality of $v$ in conjunction with \eqref{growth}, getting
\begin{align} \label{e2}
\tx{II} & = c\mint_{\br}|Du|^{p-1}|u-v| \dx + c\mint_{\br} a(x)^{\frac{1}{q}}|u-v|a(x)^{\frac{q-1}{q}}|Du|^{q-1} \dx \notag \\ & \leq c\left( \mint_{\br} 
|Du|^p \dx\right)^{\frac{p-1}{p}} \left (\mint_{\br} |u-v|^{p} \dx \right )^{\frac{1}{p}} + c\left( \mint_{\br} 
a(x)|Du|^q \dx\right)^{\frac{q-1}{q}} \left (\mint_{\br} a(x)|u-v|^{q} \dx \right )^{\frac{1}{q}} \notag \\ & \leq cr\left( \mint_{\br} 
|Du|^p \dx\right)^{\frac{p-1}{p}} \left (\mint_{\br} \tx{H}\left (x,\frac{u-v}{r}\right ) \dx \right )^{\frac{1}{p}} + cr^{\frac{p}{q}}\left( \mint_{\br} 
a(x)|Du|^q \dx\right)^{\frac{q-1}{q}} \left (\mint_{\br} \tx{H}\left (x,\frac{u-v}{r}\right )\dx \right )^{\frac{1}{q}} \notag
\\ &  \leq cr^{\frac{p}{q}} \mint_{\br}\ff(x,Du) \dx,
\end{align}
\noindent
with $c \equiv c(\tx{data}_1)$. As regards term $\tx{I}$, using $(\ref{ellipticity})_2$
 and Young's inequality, we have
\begin{align*}
    \snr{\tx{I}} & \leq c\mint_{\br}\snr{\partial_z \tilde{\g}(x,Du) - \partial_z \g(x,u,Du)}\snr{Du-Dv} \dx \\ & \leq c  \mint_{\br}\left (r + \snr{(u)_{\br}-u} \right )\snr{Du}^{p-1}\snr{Du-Dv} \dx \\ & \quad + c  \mint_{\br} a(x)\left (r + \snr{(u)_{\br}-u} \right )\snr{Du}^{q-1}\snr{Du-Dv} \dx  \\ & \leq c \left (\lambda \mint_{\br} \snr{Du-Dv}^p \dx + \lambda^{1-p} r^{\frac{p}{p-1}}\mint_{\br}\snr{Du}^{p} \dx + \lambda^{1-p}\mint_{\br}\snr{(u)_{\br} -u}^{\frac{p}{p-1}} \snr{Du}^p \dx \right . \\ & \quad +  \tilde{\lambda} \mint_{\br} a(x)\snr{Du-Dv}^q \dx + \tilde{\lambda}^{1-q}r^{\frac{q}{q-1}} \mint_{\br}a(x)\snr{Du}^{q} \dx \\ & 
    \qquad + \left . \tilde{\lambda}^{1-q}\mint_{\br}a(x)\snr{(u)_{\br} -u}^{\frac{q}{q-1}} \snr{Du}^q \dx \right ),
\end{align*}
with $c \equiv c(\tx{data}_1)$. Using \eqref{uu}, and the ideas to obtain \eqref{e1}, we get 
\begin{align} \label{e3}
    \snr{\tx{I}} & \leq c   \left ((\lambda + \tilde{\lambda}) \mint_{\br}\ff(x,Du-Dv) \dx + (\lambda^{1-p} +\tilde{\lambda}^{1-q})r^{\frac{q}{q-1}}\mint_{\brr} \ff(x,Du) \dx \right . \notag \\ & \qquad  + (\lambda^{1-p} + \tilde{\lambda}^{1-q})\left .\mint_{\br} \snr{(u)_{\br} -u}\ff(x,Du) \dx \right )\notag \\ & \leq c\left ( (\lambda + \tilde{\lambda}) \mint_{\br}\ff(x,Du-Dv) \dx +(\lambda^{1-p} + \tilde{\lambda}^{1-q})r\mint_{\brr}\ff(x,Du) \dx\right .\notag \\ & \qquad \left .+(\lambda^{1-p} + \tilde{\lambda}^{1-q}) \mint_{\brr}\ff(x,Du) \dx  \left [ r\left (\tx{H}^-_{\brr}\right )^{-1}\left ( \mint_{\brr} \ff(x,Du) \dx\right )\right ]^{\frac{\delta}{1+\delta}}
    \right ),
\end{align}
with $c$ as above. Taking $\lambda$ and $\tilde{\lambda}$ small enough, putting \eqref{e1}-\eqref{e3} in \eqref{3eq} and using \eqref{uv2}, we get
\begin{align} 
      \mint_{\br} \mathcal{V}^2(x,Du,Dv) \dx & \leq c\mint_{\brr}\ff(x,Du) \dx  \left [ r\left (\tx{H}^-_{\brr}\right )^{-1}\left ( \mint_{\brr} \ff(x,Du) \dx\right )\right ]^{\frac{\delta}{1+\delta}}+ cr^{\frac{p}{q}}\mint_{\brr}\ff(x,Du) \dx  \label{torec}  \\ &  \leq c\varepsilon^{\frac{\delta}{1+\delta}} \mint_{\brr}\ff(x,Du) \dx +cr^{\frac{p}{q}} \mint_{\brr}\ff(x,Du) \dx  \notag \\ & \leq c\left ( \varepsilon^{\frac{\delta}{1+\delta}} + r^{\frac{p}{q}}\right )\mint_{\brr}\ff(x,Du) \dx,\label{eq4}
\end{align}
where $c \equiv c(\tx{data}_1)$. Now, we observe that by \eqref{morrey}, \eqref{eq4} and the minimality of $v$ together with \eqref{growth}, for $\tau,\sigma \in (0,1/4)$, we have 
\begin{align} \label{sigma}
    \mint_{\tx{B}_{2\tau r}} \ff(x,Du) \dx & \leq c \left (\mint_{\tx{B}_{2\tau r}} \mathcal{V}^2(x,Du,Dv) \dx + \mint_{\tx{B}_{2\tau r}} \tx{H}(x,Dv) \dx \right ) \notag \\ & \leq c \left (  \tau^{-n}\mint_{\tx{B}_{r}} \mathcal{V}^2(x,Du,Dv) \dx + \tau^{-\sigma} \mint_{\tx{B}_{r/2}} \ff(x,Dv) \dx\right ) \notag \\ & \leq c \left ( \tau^{-n} \left ( \varepsilon^{\frac{\delta}{1+\delta}} + r^{\frac{p}{q}} \right )\mint_{\tx{B}_{2 r}} \ff(x,Du) \dx + \tau^{-\sigma} \mint_{\bb_{2r}} \ff(x,Du) \dx\right ) \notag \\ & = c \tau^{\sigma} \left ( \tau^{-n-\sigma} \left ( \varepsilon^{\frac{\delta}{1+\delta}} + r^{\frac{p}{q}} \right ) + \tau^{-2\sigma} \right )\mint_{\bb_{2r}}\ff(x,Du) \dx,
\end{align}
with $c \equiv c(\tx{data}_1)$. Note that, in order to apply \eqref{morrey}, we transformed the integral \( \int \tilde{\g}(x, D w) \dx \) into the standard form \( \int \ff(x, D w) \dx \), see \eqref{cv}. Let us now choose $\tau \in (0,1/4)$ and $\varepsilon \in (0,1)$ and $r_* \in (0,1)$ such that 
\eqn{varepsilon}
$$ c\tau^\sigma < 1/6, \ \varepsilon < \min \left \{ \tau^{(n-\sigma)(\frac{1+\delta}{\delta})} \right \}, \ r_*  <   \tau^{(n-\sigma)\frac{q}{p}}$$
note that $\tau,\varepsilon, r_* \equiv \tau, \varepsilon, r_*(\tx{data}_1,\sigma)$. Then, recalling that $r \in (0,r_*/2)$, inequality \eqref{sigma} reads as
\begin{equation}\label{nsigma}
\int_{\tx{B}_{2\tau r}} \ff(x,Du)\dx \overset{\eqref{varepsilon}}{\leq} \tau^{n-2\sigma}\int_{\brr}\ff(x,Du) \dx.
\end{equation}
Now, defining
$$ \mathcal{E}(u,\Omega) := \mint_{\Omega} \ff(x,Du) \dx,$$
we observe that
\begin{align*}
    \mathcal{E}(u,\tx{B}_{2 \tau r}) \overset{\eqref{nsigma}}{\leq} \tau^{-2\sigma}\mathcal{E}(u, \brr) \overset{\eqref{smallness2}}{<} \tau^{1-2\sigma}\tau^{-1}\ff^{-}_{\brr}\left ( \frac{\varepsilon}{2r}\right ) \overset{\eqref{suiraggi}}{\leq} \tau^{-1}\ff^{-}_{\tx{B}_{2\tau r}}\left ( \frac{\varepsilon}{2r}\right ) \leq \ff^{-}_{\tx{B}_{2\tau r}}\left ( \frac{\varepsilon}{2\tau r}\right ).
\end{align*}
Hence, if we choose $\tau,\varepsilon$ and $r_*$ as in \eqref{varepsilon}, and the smallness condition \eqref{smallness2} is satisfied on the ball $\brr$, then it is also satisfied on $\tx{B}_{2\tau r}$. This allows us to apply Step 1 on the smaller ball $\tx{B}_{2\tau r}$ in place of $\tx{B}_{2r}$, yielding
$$\mathcal{E}(u, \tx{B}_{2\tau^2 r})\leq \tau^{-2\sigma} \mathcal{E}(u, \tx{B}_{2\tau r}).$$
By iterating the argument over the family of shrinking balls $\{\tx{B}_{\tau^j r}\}$, we obtain
$$ \mathcal{E}(u,\tx{B}_{2\tau^j r}) < \ff^-_{\tx{B}_{2\tau^j r}} \left ( \frac{\varepsilon}{2\tau^j r} \right )$$
and
$$ \int_{\tx{B}_{2\tau^j r}} \ff(x,Du) \dx < \tau^{(n-2\sigma)j} \int_{\tx{B}_{2r}}\ff(x,Du) \dx,$$
for every $j \in \mathbb{N}$. Then, using a standard interpolation argument, we conclude that
\eqn{rho}
$$ \int_{\tx{B}_{\rr}(\tilde{x})} \ff(x,Du) \dx \leq c \left ( \frac{\rr}{r}\right )^{n-2\sigma} \int_{\tx{B}_{2r}(\tilde{x})} \ff(x,Du) \dx,$$
for all $\rr \leq 2r$, with $c$ as above. Although the inequality above was obtained assuming $\sigma \in (0,1/4)$, it can be readily extended to hold for all $\sigma \in (0,1)$. Now, 
we observe that the following functions 
$$ \tilde{x} \mapsto \mint_{\bb_{2r}(\tilde{x})}\ff(x,Du) \dx \qquad\text{  and  } \qquad\tilde{x} \mapsto\ff^-_{\brr(\tilde{x})}\left (\frac{\varepsilon}{2r}\right ),$$ 
are continuous. This is due to the absolute continuity of the integral and to the observations made in Section \ref{properties}. As a result, once \( \sigma \) is fixed, if condition \eqref{smallness2} is satisfied at some point \( \tilde{x} \in \Omega \), then there exists a ball \( \tx{B}_{4r_{\tilde{x}}}(\tilde{x}) \) such that, if \( y \in \tx{B}_{4r_{\tilde{x}}}(\tilde{x}) \), one has
\eqn{f5}
$$
\mint_{\tx{B}_{2r}(y)} \ff(x, Du) \dx < \ff^-_{\tx{B}_{2r}(y)}\left ( \frac{\varepsilon}{2r}\right ).
$$
This implies that inequality \eqref{rho} holds with \( y \) replacing \( \tilde{x} \), and with the same constant \( c \), for all \( y \in \tx{B}_{4r_{\tilde{x}}}(\tilde{x}) \).
% By a standard criterion for Hölder continuity, we then obtain that \( u \in C^{0,\gamma}(B_{r_{\tilde{x}}}(\tilde{x})) \), where \( \gamma = 1 - \frac{2\sigma}{p} \). Let us see in detail this part. 
Now, let us fix $(0,1)\ni\gamma=1-2\sigma/p$, with $\sigma\in(0,1)$. Let $2\tilde{r}_{\tilde{x}}$ be the largest radius such that condition  \eqref{smallness2} is satisfied with $r \equiv \tilde{r}_{\tilde{x}}$, $\tx{B}_{4\tilde{r}_{\tilde{x}}}(\tilde{x}) \Subset \Omega, \tilde{r}_{\tilde{x}} < r_*/4$. Then, by \eqref{rho} and \eqref{asteriscoNotes}, we have that
\begin{align*}
    \int_{\tx{B}_\rr (\tilde{x})} \ff(x,Du) \dx & \leq c \left ( \frac{\rr}{\tilde{r}_{\tilde{x}}}\right )^{n-p+p\gamma} \int_{\tx{B}_{2\tilde{r}_{\tilde{x}}}(\tilde{x})} \ff(x,Du)\dx \\ & \leq c \rr^{n-p+p\gamma} \left ( (\tilde{r}_{\tilde{x}})^{p-p\gamma} \mint_{\tx{B}_{4\tilde{r}_{\tilde{x}}}(\tilde{x})} \frac{|u|^p}{ (\tilde{r}_{\tilde{x}})^{p}} + \nr{a}_{L^\infty} \frac{|u|^q}{ (\tilde{r}_{\tilde{x}})^{q}} \dx\right ),
\end{align*}
and so, by Poincaré's inequality, we get
$$ \mint_{\tx{B}_{\rr}(\tilde{x})} |u-(u)_{\tx{B}_\rr}|^p \dx \leq \frac{c}{(\tilde{r}_{\tilde{x}})^{q-p+p\gamma}}\rr^{p\gamma},$$
with $c\equiv c(\tx{data}_1,\nr{a}_{L^{\infty}},\gamma)$. This estimate remains stable when $\tilde{x}$ is substituted by any $y \in \tx{B}_{4r_{\tilde{x}}}(\tilde{x})$, as established in \eqref{f5}. Consequently, applying a standard integral characterization of H\"older continuity, we deduce that for every $\gamma \in (0,1)$, the following holds 
\begin{equation} \label{holdu}
[u]_{0,\gamma,\tx{B}_{2r_{\tilde{x}}}(\tilde{x})} \leq \frac{c}{ (\tilde{r}_{\tilde{x}})^{q/p - 1 + \gamma}}, 
\end{equation}
with $c$ as above. This ends Step 1.
\newline \noindent \textbf{Step 2:} dimension of the singular set. \newline 
Let us define
$$ \Tilde{\Omega} := \left \{ \tilde{x} \in \Omega: \exists \  \tx{B}_{r_{\tilde{x}}}(\tilde{x}) \subset \Omega \text{ such that } u \in C^{0,\gamma}(\tx{B}_{r_{\tilde{x}}}(\tilde{x})), \text{ for some } \gamma \in (0,1)  \right \},$$
and let $\Tilde{\Sigma}$ be the singular set, i.e.
$$ \Tilde{\Sigma} := \Omega \setminus \Tilde{\Omega}.$$
Then, we observe that by what proved in Step 1, we have
$$ \Tilde{\Sigma} \subset \left \{\tilde{x} \in \Omega: \limsup_{\rr \to 0} \left [\ff^{-}_{\tx{B}_{\rr}(\tilde{x})}\left ( \frac{1}{\rr}\right ) \right ]^{-1} \mint_{\tx{B}_\rr(\tilde{x})} \ff(x,Du) \dx >0 \right  \}.$$
Now, using Lemma~\ref{lemmaHI} and the fact that 
\(\tx{H}^{-}_{{\tx{B}}_\varrho(\tilde{x})}(1/\varrho) \geq \varrho^{-p}\), we obtain
\eqn{ss}
$$
\Tilde{\Sigma} \subset \left\{ \tilde{x} \in \Omega : \limsup_{\varrho \to 0} \varrho^{p(1+\delta)-n} \int_{{\tx{B}}_\varrho(\tilde{x})} \tx{H}(x, Du)^{1+\delta} \dx > 0\right\}.
$$
Recalling that we are dealing with $p(1+\delta)\leq n$, by \cite[Proposition 2.7]{giusti}, we have
\eqn{hp}
$$\dim_{\mathscr{H}}(\Tilde{\Sigma}) \leq n - p -p\delta \quad \text{and so} \quad \mathscr{H}^{n-p}(\Tilde{\Sigma}) = 0.$$ In particular, $\snr{\Tilde{\Sigma}}=0$. In the following remark we give a precise description of the singular set $\tilde{\Sigma}$ and of $\tilde{\Omega}$.
\begin{remark} \label{rm3.1}
\normalfont
We observe that the singular set $\tilde{\Sigma} = \Omega \setminus \tilde{\Omega}$ takes exactly the form 
\eqn{ssigma}
$$\left \{ \tilde{x} \in \Omega: \limsup_{\rr \to 0} \left [\ff^{-}_{\tx{B}_{\rr}(\tilde{x})}\left ( \frac{1}{\rr}\right ) \right ]^{-1} \mint_{\tx{B}_\rr(\tilde{x})} \ff(x,Du) \dx >0\right \}.$$
Indeed, calling $\Sigma$ the set \eqref{ssigma}, from Step 1  it follows that $\tilde{\Sigma} \subset \Sigma$.  In order to prove $\Sigma\subset \tilde{\Sigma}$, we take by contradiction $\tilde{x} \in \Sigma \setminus \tilde{\Sigma}$. Then, there exists \(\gamma \in (0,1)\) such that \( u \in C^{0,\gamma}(\tx{B}_{r_{\tilde{x}}}(\tilde{x})) \) for some \( \tx{B}_{r_{\tilde{x}}}(\tilde{x}) \). Now, applying \eqref{caccioppoli2} on \(\varrho \leq r_{\tilde{x}}/2\) and \eqref{suiraggi}, we obtain:
\eqn{rm}
$$
\mint_{\tx{B}_\varrho(\tilde{x})} \tx{H}(x, Du) \dx \leq c \mint_{\tx{B}_{2\varrho(\tilde{x})}} \tx{H}^{-}_{\tx{B}_{2\varrho(\tilde{x})}}\left( \frac{u - (u)_{\tx{B}_{2\varrho}(\tilde{x})}}{2\varrho} \right) \dx \leq c \tx{H}^{-}_{\tx{B}_{2\varrho}(\tilde{x})}(\varrho^{\gamma-1})\leq c \tx{H}^{-}_{\tx{B}_{\varrho}(\tilde{x})}(\varrho^{\gamma-1}),
$$
and so
\eqn{so}
$$
\left[ \tx{H}^{-}_{\tx{B}_\varrho(\tilde{x})}\left( \frac{1}{\varrho} \right) \right]^{-1} \mint_{\tx{B}_\varrho(\tilde{x})} \tx{H}(x, Du) \, \dx \leq c \varrho^{\gamma p},
$$
where \( c \equiv c(\tx{data},[u]_{0,\gamma}) \). Taking the limit for \(\varrho \to 0\), it follows that \( \tilde{x} \notin \Sigma \), a contradiction. Hence, $\Tilde{\Sigma} = \Sigma$. Now, we observe that actually we have
\eqn{pdo}
$$ \tilde{\Omega} := \left \{ \tilde{x} \in \Omega: \exists \  \tx{B}_{r_{\tilde{x}}}(\tilde{x}) \subset \Omega \text{ such that } u \in C^{0,\gamma}(\tx{B}_{r_{\tilde{x}}}(\tilde{x})), \text{ for all } \gamma \in (0,1)   \right \}.$$
Indeed, calling $\hat{\Omega}$ the set in the right-hand side of \eqref{pdo}, by Step 1 it holds $\hat{\Omega} \subset \tilde{\Omega}$. Taking $\tilde{x} \in \tilde{\Omega}$ then, there exists $\tx{B}_{r_{\tilde{x}}}(\tilde{x}) \subset \Omega$ such that $u \in C^{0,\gamma_1}(\tx{B}_{r_{\tilde{x}}}(\tilde{x}))$ for some $\gamma_1 <1$, therefore \eqref{so} holds with $\gamma_1$ in place of $\gamma$. Now, fix $\gamma$ and take $\rr$ small enough in \eqref{so} so that \eqref{smallness2} is satisfied for an $\varepsilon \equiv \varepsilon(\tx{data},\gamma)$. Hence, $u$ is $\gamma$-H\"older continuous in a neighbourhood of $\tilde{x}$. Since $\gamma$ is arbitrarily it follows that $\tilde{x} \in \hat{\Omega}$ and then $\tilde{\Omega}=\hat{\Omega}$, so \eqref{pdo} is proved. 
\end{remark}
\noindent
\noindent\textbf{Step 3:} partial H\"older regularity of the gradient. \newline
So far, we have shown that for every open subset $\Omega_0 \Subset \Tilde{\Omega} $ and every $\gamma \in (0,1)$ the minimizer $u \in C^{0,\gamma}(\Omega_0, \overline{\mm})$, as a consequence of a standard covering argument. Therefore, for all $\tx{B}_{4\rr} \Subset \Omega_0$, with $\rr\leq 1/64$, by \eqref{caccioppoli2} and \eqref{suiraggi}, for all $\gamma \in (0,1)$, it holds
\eqn{rogamma}
$$ \mint_{\tx{B}_{2\rr}} \ff(x,Du) \dx \leq c \mint_{\tx{B}_{4\rr}} \ff^{-}_{\tx{B}_{4\varrho}}\left(\frac{u-(u)_{\bb_{4\rr}}}{4\rr}\right) \dx \leq c \tx{H}^{-}_{\tx{B}_{2\varrho}}(\varrho^{\gamma-1})\leq c \tx{H}^{-}_{\tx{B}_{\varrho}}(\varrho^{\gamma-1}),$$
with $c\equiv c(\tx{data},\gamma,\Omega_0)$, and so,
\eqn{2rhogamma}
$$ \mint_{\tx{B}_{2\rr}} \ff(x,Du) \leq c\rr^{(\gamma-1)q},$$
for  \( c \equiv c( \tx{data},\nr{a}_{L^{\infty}},\gamma,\Omega_0) \). Now, fix a ball $\tx{B}_{4s\frac{\sqrt{L}}{\sqrt{\ell}}} \equiv \tx{B}_{4s\frac{\sqrt{L}}{\sqrt{\ell}}}(\tilde{x})$, with $\max\{s,s/\sqrt{\ell}\} \leq 1/64$, such that
\eqn{s0}
$$\tx{B}_{4s\frac{\sqrt{L}}{\sqrt{\ell}}} \Subset \Omega_0.$$
 Let $v \in u+W_0^{1,\ff}(\tx{B}_{s}, \R^N)$ be the solution of the Dirichlet problem \eqref{dirichlet}, with $\br =   \tx{B}_s \equiv \tx{B}_s(\tilde{x})$. Now, we recall that  since $\tilde{g}, \tilde{B}$ in \eqref{ng} are constant, symmetric, uniformly elliptic  matrices, the functional $\int_{\tx{B}_{s}}\tilde{\g}(x,Dw) \dx$ can be transformed, up to constants, into the following
\eqn{newG}
$$ \int_{\tx{E}} (|D\tilde{w}|^p + \tilde{a}(y)|D\tilde{w}|^q) \dy,$$
where $\tx{E}$ is the ellipsoid obtained from $\tx{B}_{s}$ after  the coordinate change. More precisely, the transformation is given by
\eqn{cv}
$$
A := \tilde{g}^{-1/2} \in \mathbb{R}^{n \times n}, \quad C := \tilde{B}^{-1/2} \in \mathbb{R}^{N \times N}, \quad y := A x, \quad \tilde{w}(y) := C^{-1} w(A^{-1} y), \quad \tilde{a}(y) :=  a(A^{-1} y).
$$
Under this change of variables, the ellipsoid is \( \tx{E} := A(\tx{B}_s) \subset \mathbb{R}^n \), and satisfies
\[
\tx{B}_{\frac{s}{\sqrt{L}}} \subset \tx{E} \subset \tx{B}_{\frac{s}{\sqrt{\ell}}}.
\]
We also call $\tilde{u}$ and $\tilde{v}$ the new functions obtained from $u$ and $v$, respectively, and denote by $\Tilde{\Omega}_0$ the domain obtained from $\Omega_0$. Let ${\tx{B}}_{\tilde{s}} \equiv \tx{B}_{\frac{s}{\sqrt{L}}}$; from now on we keep denoting the variable of integration with $x$ and the weight $\tilde{a}$ with $a$. We observe that since $v \in u + W^{1,\tx{H}}_0(\tx{B}_{s}, \R^N)$ is a solution of the Dirichlet problem \eqref{dirichlet} then the minimizer  $\tilde{v} \in \tilde{u} + W_0^{1,\tx{H}}(\tx{E}, \R^N)$ of \eqref{newG}, satisfies
\eqn{el2}
$$ \int_{\tx{B}_{\tilde{s}}} \partial_z \tx{H} (x,D\tilde{v}) \cdot D\varphi \dx =0,$$
for all $\varphi \in C^{\infty}_c(\tx{B}_{\tilde{s}}, \R^N)$. Now, let us select $x_0 \in \overline{\tx{B}}_{\tilde{s}}$ such that
$$ a(x_0) = \inf_{x \in \tx{B}_{\tilde{s}}} a(x).$$
% and, bearing in mind the notation in \eqref{ng}, consider
% $$ \tilde{\g}_0(z) := \tilde{\g}_0(x_0,z)= \tilde{\g}_{\tx{p}}(z) + a(x_0)\tilde{\g}_{\tx{q}}(z).$$
We want to show that $\tv$ fits the assumption of Lemma \ref{approximation} with the choice of $\tx{H}_0(\cdot)= \ff^-_{\tx{B}_{\tilde{s}}}(\cdot)$. Of course,
$$ \int_{\tx{B}_{\tilde{s}}} \tx{H}^-_{\tx{B}_{\tilde{s}}}(D\tv) \dx < +\infty,$$
moreover, looking at the proof of \eqref{hi}, we see that assumption \eqref{hiapprox} is satisfied since $\tv$ is a minimizer of \eqref{newG}. Now, it remains to show \eqref{secondaipotesi}. The analysis requires distinguishing two cases based on the behavior of the coefficient $a(\cdot)$. The $p$-phase occurs when
\eqn{pphase}
$$
a(x_0) \leq 4[a]_{0,\alpha}{\tilde{s}}^{\alpha-t}
$$
where $t := \alpha + (\gamma-1)(q-p) > 0$ by \eqref{pq}. The $(p,q)$-phase emerges in the complementary case
\eqn{pqphase}
$$
a(x_0) > 4[a]_{0,\alpha}{{\tilde{s}}}^{\alpha-t}.
$$
Then, it follows that 
\begin{equation}
\label{eq:phase_bounds}
\displaystyle \begin{cases}
\displaystyle\sup_{\tx{B}_{\tilde{s}}} a(x) \leq 6[a]_{0,\alpha}{{\tilde{s}}}^{\alpha-t} & \text{in the } p\text{-phase} \\
\displaystyle \sup_{\tx{B}_{\tilde{s}}} a(x) \leq \frac{3}{2}a(x_0) & \text{in the } (p,q)\text{-phase}.
\end{cases}
\end{equation}
Note that the two phases depend on the number $\gamma \in (0,1)$ that we will choose later. Now, for all $\varphi \in C^{\infty}_c(\tx{B}_{{\tilde{s}}/2},\R^N)$, it holds
\begin{align} \label{i0}
\left | \mint_{\tx{B}_{{\tilde{s}}/2} } \partial_z \ff^-_{\tx{B}_{\tilde{s}}}(D\tv) \cdot D\varphi \dx \right | &\mathrel{\eqmathbox{\overset{\mathrm{\eqref{el2}}}{=}}} \left |\mint_{\tx{B}_{{\tilde{s}}/2}} \left [ \partial_z \ff^-_{\tx{B}_{\tilde{s}}} (D\tv) - \partial_z \tx{H}(x,D\tv)\right ] \cdot D\varphi \dx \right | \notag \\ & \mathrel{\eqmathbox{\overset{\mathrm{}}{\leq}}}  c[a]_{0,\alpha}{{\tilde{s}}}^{\alpha} \nr{D\varphi}_{L^{\infty}(\tx{B}_{\tilde{s}/2})} \mint_{\tx{B}_{{\tilde{s}}/2}}|D\tv|^{q-1} \dx =: \Tilde{\tx{I}}.
\end{align}
In order to estimate $\tilde{\tx{I}}$ we have to distinguish between the the $p$-phase \eqref{pphase} and the the $(p,q)$-phase \eqref{pqphase}. We first observe that by the choice made in  \eqref{s0}, it holds $\tx{B}_{\frac{4s}{\sqrt{\ell}}} \Subset \tilde{\Omega}_0$ and \eqref{rogamma} holds true for $\tilde{u}$ in place of $u$ and $\rr=s/\sqrt{\ell}$.
Now, when we are in \eqref{pphase}, using the minimality of $\tv$, \eqref{rogamma} and the definition of $t$, it holds
\eqn{ded}
$$ \mint_{\tx{B}_{\tilde{s}}} \tx{H}(x,D\tv) \dx \leq \mint_{\tx{E}}\tx{H}(x,D\tv) \dx \leq \mint_{\tx{E}} \tx{H}(x,D\tu) \dx\leq  c  
\mint_{\tx{B}_{\frac{s}{\sqrt{\ell}}}} \tx{H}(x,D\tu) \dx  \leq c  {s}^{(\gamma-1)p},$$
with $c \equiv c(\tx{data}, \nr{a}_{L^{\infty}},\gamma,\Omega_0).$ Then,  by H\"older's inequality, applying the definition of $t$, \eqref{ded}, and recalling that $\tilde{s}=s/\sqrt{L}$, we have
\begin{align} \label{i1}
\tilde{\tx{I}} \leq c{{\tilde{s}}}^t 	{\tilde{s}}^{\alpha-t}\nr{D\varphi}_{L^{\infty}(\tx{B}_{{\tilde{s}}/2})} \left ( \mint_{\tx{B}_{{\tilde{s}}/2}} |D\tv|^p \dx \right )^{\frac{q-p}{p}}\left ( \mint_{\tx{B}_{{\tilde{s}}/2}} |D\tv|^p \dx\right )^{\frac{p-1}{p}} \leq c {\tilde{s}}^t \mint_{\tx{B}_{\tilde{s}}}\left ( |D\tv|^p + \nr{D\varphi}_{L^{\infty}(\tx{B}_{{\tilde{s}}/2})}^p\right ) \dx,
\end{align}
note that we have used the fact that $q<p+1$, with $c$ as above. While in the $(p,q)$-phase \eqref{pqphase}, by Young's inequality, we have
\begin{align} \label{i2}
\tilde{\tx{I}} &= c[a]_{0,\alpha} {\tilde{s}}^{\alpha-t}{\tilde{s}}^t \nr{D\varphi}_{L^{\infty}(\tx{B}_{{\tilde{s}}/2})} \mint_{\tx{B}_{{\tilde{s}}/2}}|D\tv|^{q-1}  \dx \notag \\ & \leq c{{\tilde{s}}}^t\nr{D\varphi}_{L^{\infty}(\tx{B}_{{\tilde{s}}/2})}[a(x_0)]^{\frac{1}{q}}\mint_{\tx{B}_{\tilde{s}}}[a(x_0)]^{\frac{q-1}{q}}|D\tv|^{q-1} \dx \notag \\ & \leq c{{\tilde{s}}}^t \mint_{\tx{B}_{\tilde{s}}} \left ( a(x_0)|D\tv|^q +a(x_0)\nr{D\varphi}_{L^{\infty}(\tx{B}_{{\tilde{s}}/2})}^q\right ) \dx,
\end{align}
with $c \equiv c(n,N,\ell,L,p,q,\alpha,[a]_{0,\alpha})$. Then, using \eqref{i1} and \eqref{i2} in \eqref{i0}, we get
$$ \left | \mint_{\tx{B}_{{\tilde{s}}/2} } \partial_z \ff^-_{\tx{B}_{\tilde{s}}}(D\tv) \cdot D\varphi \dx  \right | \leq c {\tilde{s}}^t \mint_{\tx{B}_{\tilde{s}}} \left ( \tx{H}^-_{\tx{B}_{\tilde{s}}}(D\tv)+\tx{H}^-_{\tx{B}_{\tilde{s}}}(\nr{D\varphi}_{L^{\infty}(\tx{B}_{{\tilde{s}}/2})}) \right ) \dx,$$
with $c \equiv c(\tx{data},\nr{a}_{L^{\infty}},\gamma,\Omega_0)$. So, we are in position to apply Lemma \ref{approximation}: there exists $\tilde{h} \in \tv +W^{1,\tx{H}^{-}_{\tx{B}_{\tilde{s}}}}_0(\tx{B}_{{\tilde{s}}/2} , \R^N)$, minimizer of
\eqn{th}
$$\tv +W^{1,\tx{H}^{-}_{\tx{B}_{\tilde{s}}}}_0(\tx{B}_{{\tilde{s}}/2} , \R^N) \ni w \mapsto \int_{\tx{B}_{{\tilde{s}}/2} }\ff^-_{\tx{B}_{\tilde{s}}}(Dw) \dx,$$
satisfying
\eqn{vh}
$$ \mint_{\tx{B}_{{\tilde{s}}/2} }\mathcal{V}^2(D\tv, D\tilde{h}, \tx{B}_{\tilde{s}}) \dx \leq c{{\tilde{s}}}^m \mint_{\tx{B}_{\tilde{s}}}\ff^-_{\tx{B}_{\tilde{s}}}(D\tv) \dx\leq c{{\tilde{s}}}^m \mint_{\tx{B}_{\tilde{s}}} \tx{H}(x,D\tv) \dx \leq  c{{\tilde{s}}}^m 
\mint_{\tx{B}_{\frac{s}{\sqrt{\ell}}}} \tx{H}(x,D\tu) \dx \leq c{s}^{m+(\gamma-1)q},$$
where for the last inequality we used \eqref{2rhogamma} that holds true  for $\tilde{u}$ in place of $u$ and $\rr=s/\sqrt{\ell}$ (recall that $\tx{B}_{\frac{4s}{\sqrt{\ell}}}  \Subset \tilde{\Omega}_0$ by \eqref{s0}); with $m \equiv m(n,N,\ell,L,p,q,\alpha)$ and $c \equiv c(\tx{data},\nr{a}_{L^{\infty}},\gamma,\Omega_0)$. Now, since $\tilde{h}$ minimizes \eqref{th}, it holds
\eqn{r1}
$$ \sup_{\tx{B}_{\tilde{s}/4}} \ff^-_{_{\tx{B}_{\tilde{s}}}} (D\tilde{h}) \leq c \mint_{\tx{B}_{\tilde{s}/2}} \ff^-_{_{\tx{B}_{\tilde{s}}}} (D\tilde{h}) \dx,$$
for a constant $c \equiv c(n,N,\ell,L,p,q)$, see \cite[Lemma 5.8]{DSV}. Now, recalling the definitions \eqref{V1} and \cite[(1.3)]{DSV}, we can apply \cite[Theorem 6.4]{DSV} that gives
\begin{align} \label{r2}
    & \mint_{\tx{B}_\rr}\left (|V_p(D\tilde{h})-(V_p(D\tilde{h}))_{\tx{B}_\rr}|^2 + \inf_{\tx{B}_{\tilde{s}}}a(x)|V_q(D\tilde{h})-(V_q(D\tilde{h}))_{\tx{B}_\rr}|^2\right) \dx \notag \\ & \qquad \leq c\left (\frac{\rr}{\tilde{s}}\right )^{2\mu} \mint_{\tx{B}_{\tilde{s}/2}} \left (|V_p(D\tilde{h})-(V_p(D\tilde{h}))_{\tx{B}_{\tilde{s}/2}}|^2 + \inf_{\tx{B}_{\tilde{s}}}a(x)|V_q(D\tilde h)-(V_q(D\tilde{h})_{\tx{B}_{\tilde{s}/2}}|^2\right) \dx, 
\end{align}
for $1 \leq c(n,N,p,q,\ell,L)$, $(0,1/2) \ni \mu \equiv \mu(n,N,p,q,\ell,L)$ and whenever $\tx{B}_\rr \subset \tx{B}_{\tilde{s}/2}$ is concentric to $ \tx{B}_{\tilde{s}/2}$. Combining \eqref{r1} and \eqref{r2} with \eqref{v} and following \cite[Theorem 3.1]{BCM1}, it holds
\eqn{forh}
$$\mint_{\tx{B}_{\rr}} \ff^-_{\tx{B}_{\tilde{s}}}(D\tilde{h} -(D\tilde{h})_{\bb_{\rr}}) \dx \leq c \left ( \frac{\rr}{{\tilde{s}}}\right )^\mu \mint_{\tx{B}_{\tilde{s}/2}} \tx{H}^{-}_{\tx{B}_{\tilde{s}}}(D\tilde{h}) \dx  \leq c \left ( \frac{\rr}{{\tilde{s}}}\right )^\mu \mint_{\tx{B}_{\tilde{s}}} \ff(x,D\tilde{u}) \dx \overset{\eqref{2rhogamma}}{\leq} c \left ( \frac{\rr}{{\tilde{s}}}\right )^\mu \tilde{s}^{(\gamma -1)q},$$
for  $c \equiv c(\tx{data},\nr{a}_{L^{\infty}},\gamma,\Omega_0)$ and $\mu$ as above. Now, by \eqref{torec}, we have
\begin{align*}
  \mint_{\tx{B}_s} \mathcal{V}^2(x,Du, Dv)\dx & \mathrel{\eqmathbox{\overset{\mathrm{}}{\leq}}} c \mint_{\tx{B}_{2s}}\ff(x,Du) \dx\left ( s(\ff^-_{\tx{B}_{2s}})^{-1}\left (\mint_{\tx{B}_{2s}}\ff(x,Du) \dx\right ) \right )^{\frac{\delta}{1+\delta}} \\ & \qquad \mathrel{\eqmathbox{\overset{\mathrm{}}{+}}}cs^{\frac{p}{q}}\mint_{\tx{B}_{2s}}\ff(x,Du) \dx   \\ & \mathrel{\eqmathbox{\overset{\mathrm{\eqref{rogamma} \eqref{2rhogamma}}}{\leq}}} c
 s^{(\gamma-1)q}\left (s(\ff^-_{\tx{B}_{2s}})^{-1}\left (\ff^-_{\tx{B}_{2s}}(s^{\gamma-1}) \right)\right)^{\frac{\delta}{1+\delta}} + cs^{(\gamma-1)q+\frac{p}{q}}  \\ & \mathrel{\eqmathbox{\overset{\mathrm{}}{\leq}}} c  s^{(\gamma-1)q}  s^{\gamma\frac{\delta}{1+\delta}} +cs^{\gamma q-q+\frac{p}{q}} \\ & \mathrel{\eqmathbox{\overset{\mathrm{}}{\leq}}} c\left(s^{\gamma\left ( \frac{\delta}{1+\delta} +q\right )-q}+s^{\gamma q-q+\frac{p}{q}}\right),
\end{align*}
with $c \equiv c(\tx{data}_1,\nr{a}_{L^{\infty}},\gamma,\Omega_0)$. One can easily verify that, applying the change of coordinates \eqref{cv}, it holds
\eqn{newV}
$$ \mint_{\tx{B}_{\tilde{s}}} \mathcal{V}^2(x,D\tu,D\tv) \dx \leq c\left( s^{\gamma\left ( \frac{\delta}{1+\delta} +q\right )-q}+ s^{\gamma q-q+\frac{p}{q}}\right).$$ Therefore, collecting  \eqref{vh}, \eqref{forh}, \eqref{newV}, recalling that $\tilde{s}= s/\sqrt{L}$, we estimate
\begin{align} \label{stimasuu}
    \mint_{\tx{B}_\rr}|D\tu -(D\tu)_{\bb_\rr}|^p \dx &\leq c\mint_{\tx{B}_\rr} \mathcal{V}^2(D\tv, D\tilde{h}, \tx{B}_\rr) \dx + c\mint_{\tx{B}_\rr} \mathcal{V}^2(x,D\tu, D\tilde{v})+ c\mint_{\tx{B}_{\rr}} \ff^-_{\tx{B}_{ \tilde{s}}}(D\tilde{h} -(D\tilde{h})_{\bb_\rr}) \dx \notag \\ & \leq c\left [\left( \frac{ \tilde{s}}{\rr}\right)^n{{ \tilde{s}}}^{q\gamma+m-q }+\left( \frac{ \tilde{s}}{\rr}\right)^n\left ( \tilde{s}^{\gamma\left ( \frac{\delta}{1+\delta} +q\right )-q}+ \tilde{s}^{\gamma q-q+\frac{p}{q}} \right ) +\left ( \frac{\rr}{{ \tilde{s}}}\right )^\mu  \tilde{s}^{(\gamma -1)q} \right ], 
\end{align}
with $c$ as above. Now, we select $\gamma$ such that
\eqn{sgamma}
$$ \gamma > \max \left \{\frac{q}{q+\frac{\delta}{1+\delta}}, 1-\frac{p}{q^2}, \frac{q-m}{q} \right \}, $$
and let
$$ \nu:= \min \left \{ m+(\gamma-1)q, \gamma \left ( \frac{\delta}{1+\delta}+q\right )-q,\gamma q-q+\frac{p}{q}  \right \}.$$
Note that $0<\nu<1$. Then \eqref{stimasuu} turns into,
\eqn{ql}
$$  \mint_{\tx{B}_\rr}|D\tu -(D\tu)_{\bb_\rr}|^p \dx \leq c\left [ \left( \frac{ \tilde{s}}{\rr}\right)^n  \tilde{s}^\nu + \left ( \frac{\rr}{{ \tilde{s}}}\right )^\mu  \tilde{s}^{(\gamma -1)q}\right].$$
Now, we take
$$ \gamma \geq 1-\frac{\nu\mu}{2nq},$$
and
$$ \rr:=  \frac{ \tilde{s}^{1+\theta}}{8} \quad \text{ with } \quad \theta := \frac{(1-\gamma)q + \nu}{\mu +n},$$
therefore \eqref{ql} reads as
\eqn{last}
$$  \mint_{\tx{B}_\rr}|D\tu -(D\tu)_{\tx{B}_\rr}|^p \dx \leq c \rr^{\frac{\nu -\theta n}{1+\theta}} \leq c\rr^{\beta p}, \qquad \beta :=\frac{\nu \mu}{4p(n+\mu)},$$
with $c \equiv c(\tx{data}_1,\nr{a}_{L^{\infty}},\gamma,\Omega_0)$. Hence, since $\Omega_0$ is arbitrary, by the characterization of H\"older continuity of Campanato and Meyers and by a standard covering argument, it follows that $ D\tilde{u} \in C^{0,\beta}_{\loc}(A(\tilde{\Omega}), \R^{N\times n})$, with $\beta \equiv \beta(\tx{data})$ as in \eqref{last}. Then, applying the inverse of the transformation in \eqref{cv}, we conclude that $Du \in C^{0,\beta}_{\loc}(\tilde{\Omega} ,\R^{N\times n})$. This ends Step 3 in the case $p(1+\delta) \leq n$.

For \(p(1 + \delta) > n\), the singular set $\tilde{\Sigma}$ turns out to be empty, that is, \(\tilde{\Sigma} = \tilde{\Omega} = \Omega\), as the right-hand side of \eqref{ss} is the empty set. From Step 1 and Step 2 we then deduce that \(u \in C^{0,\gamma}_{\loc}(\Omega,\overline{ \mm})\) for every \(\gamma < 1\), and the conclusion concerning the local  H\"older continuity of \(Du\) follows exactly as in the case \(p(1 + \delta) \leq n\). Then \eqref{stat} is completely proved.
\end{proof}
\noindent
Let us now turn to the proof of Theorem \ref{maintheorem2}. 
\begin{proof}[Proof of Theorem \ref{maintheorem2}] We recall that now $\Omega$ is a bounded open subset of an $n$-dimensional Riemannian manifold. Choosing local coordinates on $\Omega$, the functional $\mathcal{H}$ in \eqref{functional} can be expressed as
\eqn{funpre}
$$
\int \left(h_{\alpha \beta} D_\alpha u \cdot D_\beta u\right)^{p/2} + \tilde{a}(x)\left(h_{\alpha \beta} D_\alpha u \cdot D_\beta u\right)^{q/2} \, \dd x,
$$
with $\tilde{a} \in C^{0,\alpha}$ and coefficients $h_{\alpha \beta}$ (independent on variable $u$) satisfying the same assumptions as $g_{\alpha \beta}$ in \eqref{ellipticity}, with $\ell = \ell_h$. Then, using equation~\eqref{euler2} with this metric on $\Omega$, the proof follows directly from the arguments already developed for Theorem~\ref{maintheorem}. We point out that now the function $\tx{f}$ in \eqref{equation2} collects all the
terms in \eqref{euler2} with the dependence on the variable $u$, namely
it contains the contribution involving $D^2\Pi(u)$ on the left-hand side
of \eqref{euler2} as well as the right-hand side term. Note also that Lemma \ref{CaccioppoliLemma} and \ref{lemmaHI} apply also in the case of a manifold $\mm$ as in \ref{item:b} or \ref{item:c}.
\end{proof}
\noindent
We now state the version corresponding to manifolds as in \ref{item:c}, whose validity is an immediate consequence of Theorem \ref{maintheorem2} since the Euler-Lagrange equation \eqref{euler2} is the same but without  the boundary term.
\begin{theorem} \label{maintheorem3}
Let $\Omega$ be a bounded open subset of an $n$-dimensional Riemannian manifold and let $u \in W^{1,1}_{\loc}(\Omega, {\mm})$ be a local minimizer of the functional $\mathcal{H}$ in \eqref{functional} under assumption \eqref{pq}, with $\mm$ as in \ref{item:c} and satisfying \eqref{M}. Then, there exists an open subset $\Tilde{\Omega} \subset \Omega$ such that
\eqref{stat} holds true,
     for some $\beta \equiv \beta(\tx{data}) \in (0,1)$. 
\end{theorem} 
\noindent
The following corollary is immediate.
\begin{corollary}
Under the assumptions of Theorem \ref{maintheorem},  let $\delta$ be the exponent arising in Lemma \eqref{lemmaHI}. Assume $p(1+\delta) > n$, then
      $$ \Omega = \tilde{\Omega}.$$
\end{corollary}
\noindent
The same corollary holds under the assumptions of Theorem \ref{maintheorem2} or Theorem \ref{maintheorem3} after introducing local coordinates on $\Omega$. The next remark describes the part of the singular set lying in the region where the modulating coefficient $a(\cdot)$ is strictly positive, in the context of Theorem \ref{maintheorem}.

\begin{remark}
\normalfont  
Let us assume that $q(1+\delta) \leq n$, where again $\delta$ is the exponent arising in Lemma \ref{lemmaHI}. We now provide a characterization of the region of singular set \( \tilde{\Sigma} \cap \{ a(x)>0\}\) in terms of its Hausdorff dimension. Let \( \tilde{x} \in \tilde{\Sigma} \) be such that \( a(\tilde{x}) > 0 \). Similarly to Step 2 in Theorem \ref{maintheorem}, take \( \varrho \) sufficiently small, from
\[
\ff^{-}_{\tx{B}_{\varrho}(\tilde{x})}\left(\frac{1}{\varrho}\right) \geq \inf_{\tx{B}_{\varrho}(\tilde{x})} a(x)\, \varrho^{-q} > 0,
\]
we have
\eqn{asigma}
$$\tilde{\Sigma} \cap \{ a(x) > 0 \} \subset \left\{ \tilde{x} \in \Omega : \limsup_{\varrho \to 0} \varrho^{q(1+\delta)-n} \int_{\tx{B}_{\varrho}(\tilde{x})} \ff(x, Du)^{1+\delta}\, \dx > 0 \right\}.$$
By applying \cite[Proposition 2.7]{giusti}, it follows that
\eqn{hq}
$$
\dim_{\mathscr{H}}(\tilde{\Sigma} \cap \{ a(x) > 0 \}) \leq n - q-q\delta \quad \text{and so}\quad \mathscr{H}^{n-q}(\tilde{\Sigma} \cap \{ a(x) > 0 \}) = 0.
$$
\end{remark}
\vspace{0.3cm}
\noindent
We conclude this section by providing the following estimates for the Hausdorff dimension of the singular set $\tilde{\Sigma}$.
\begin{corollary} \label{lcor}
Under the assumptions of Theorem \ref{maintheorem}, the following Hausdorff dimension estimates hold
\eqn{dimension}
$$ \dim_{\mathscr{H}}(\tilde{\Sigma}) \leq n-p, \qquad \dim_{\mathscr{H}}(\tilde{\Sigma} \cap \{a(x)>0\}) \leq n-q,$$
where $\tilde{\Sigma}$ is the singular set.
\end{corollary}
\begin{proof}
By \eqref{hp} and \eqref{hq}, combined with the facts that $\tilde{\Sigma} = \emptyset$ when $p(1+\delta) > n$ (by \eqref{ss}) and $\tilde{\Sigma} \cap \{a(x)>0\} = \emptyset$ when $q(1+\delta) > n$ (by \eqref{asigma}), the claim \eqref{dimension} immediately follows.
\end{proof}
\noindent
Again, the same corollary holds under the assumptions of Theorem \ref{maintheorem2} or Theorem \ref{maintheorem3} after introducing local coordinates on $\Omega$.

\section{Improved estimate for the Hausdorff measure of the singular set} \label{sec4}
\noindent
In this final section, we exploit the partial regularity result established in Theorems \ref{maintheorem}, \ref{maintheorem2}, \ref{maintheorem3}, to obtain some improved estimate for the Hausdorff measure of the singular set \( \tilde{\Sigma} := \Omega \setminus \tilde{\Omega} \).
The analysis distinguishes between two regimes, depending on the relation between the quantity \( q(1+\delta) \) and the dimension \( n \), where $\delta$ is the exponent coming up from Lemma \ref{lemmaHI}. In order to prove \eqref{2.23} we need to introduce some extra notation. First of all, for any $n$-dimensional open ball $\tx{B} \subset \Omega$ of radius $r(\tx{B}) \in (0,\infty)$ we introduce the (finite) function
$$ h_\ff(\tx{B}) := \int_{\tx{B}} \ff(x,1/r(\tx{B})) \dx.$$
Now, we use the standard Carathéodory's  construction to obtain an outer measure. We define the weighted $k$-approximating Hausdorff measure of a set $E \subset \Omega$ as
$$ \mathscr{H}_{\ff,k}(E) := \inf_{\mathcal{C}^k_E} \sum_{j} h_{\ff}(\tx{B}_j),$$
where
$$ \mathcal{C}^k_{E} := \left \{ \{ \tx{B}_j\}_{j \in \mathbb{N}} \text{ is a countable collection of balls $\tx{B}_j \subset \Omega$ covering } E \text{ such that } r(\tx{B}_j) \leq k\right \}.$$
Moreover, we define
$$ \mathscr{H}_{\ff}(E) := \lim_{k \to 0} \mathscr{H}_{\ff,k}(E).$$
Considering the functions $t \mapsto \text{ess sup}_{x \in \tx{B}} \ff(x,t)$  and $t \mapsto \essinf_{x \in \tx{B}} \ff(x,t)$ we introduce
$$  h^+_{\ff}(\tx{B}) := |\tx{B}|  \text{ess sup}_{x \in \tx{B}} \ff(x,1/r(\tx{B})), \quad h^-_{\ff}(\tx{B}) := |\tx{B}| \essinf_{x \in \tx{B}} \ff(x,1/r(\tx{B}))$$
and consequently
$$ \mathscr{H}^{\pm}_{\ff,k}(E) = \inf_{\mathcal{C}^k_E}\sum_j h^{\pm}_{\ff}(\tx{B}_j) \quad \text{and} \quad \mathscr{H}_{\ff}^{\pm}(E) = \lim_{k \to 0} \mathscr{H}^{\pm}_{\ff,k}(E).$$
The following proposition holds true, for the proof we refer to \cite[Section 6]{DFM}.
\begin{proposition} \label{prop2} Assume that $q \leq p+\alpha$ and $a(\cdot) \in C^{0,\alpha}(\Omega)$. Then, for any subset $E\subset \Omega$ and any $\delta \geq 0$  there exists a constant $c \equiv c([a]_{0,\alpha},\delta) \geq 1$ such that
    $$ \mathscr{H}^-_{\ff^{1+\delta}}(E) \leq \mathscr{H}_{\ff^{1+\delta}}(E) \leq \mathscr{H}^+_{\ff^{1+\delta}}(E) \leq c \mathscr{H}^-_{\ff^{1+\delta}}(E).$$
\end{proposition}
\noindent
Now, we are ready to prove Theorem \ref{theorem2}.
\begin{proof}[proof of Theorem \ref{theorem2}]
First of all, we observe that from Lemma \ref{lemmaHI} we have the existence of $\delta \equiv \delta(\tx{data})>0$ such that \eqref{suH} holds. The proof of \eqref{2.2} and \eqref{2.3} follow directly from Theorem~\ref{maintheorem}. In particular, \eqref{2.2} is a consequence of \eqref{hp}, while \eqref{2.3} follows from \eqref{hq}. As regards \eqref{2.23}, the proof is the same of \cite[Theorem 3]{DFM} and makes use of Proposition \ref{prop2}. The proof is complete. 
\end{proof}
\noindent
Theorem~\ref{the1.4} is a direct consequence of Theorem~\ref{theorem2} after introducing local coordinates on $\Omega$. As a further consequence, we obtain the following final result.
\begin{theorem}
Let $\Omega$ be a bounded open subset of an $n$-dimensional Riemannian manifold and        let $u \in W^{1,1}_{\loc}(\Omega, \mm)$ be a local minimizer of the functional $\mathcal{H}$ in \eqref{functional}, assume \eqref{pq}, with $\mm$ as in \ref{item:c} and satisfying \eqref{M}. Then, there exists $\delta \equiv \delta(\tx{data})>0$ such that \eqref{suH} holds. Moreover, if $q(1+\delta) \leq n$ then \eqref{2.23}-\eqref{2.3} are true.
\end{theorem}

\section*{Acknowledgements}

\noindent
We would like to thank Professor Christopher Hopper for his insightful explanations concerning the proof of Lemma 4.5 and 4.7 in \cite{H}.

This research was partly conducted while C. Pacchiano Camacho was as a Research Visitor at Uppsala University. This work was supported by a grant from Simons Foundation International SFI-MPS-T-Institutes-00011977 JS. \newline
A. Nastasi is member of the Gruppo Nazionale per l’Analisi
Matematica, la Probabilit\`{a} e le loro Applicazioni (GNAMPA) of the Istituto
Nazionale di Alta Matematica (INdAM).
A. Nastasi was partially supported by Grant \textit{D26-PREMIOGRUPPI-RIC-2024-NASTASI - RS Antonella Nastasi}  assigned by the Department of Engineering of University of Palermo. \newline
F. De Filippis has been partially supported through the INdAM - GNAMPA Project (CUP E5324001950001).

\end{document}